%% file: main.tex
\documentclass[12pts]{amsart}

\usepackage{amsmath,amsthm,amscd,amssymb, color}
\usepackage{latexsym}
\usepackage{fullpage}
\usepackage{tikz-cd}
\usepackage{mathtools}
\usepackage[toc]{appendix}

\usepackage [english]{babel}
\usepackage [autostyle, english = american]{csquotes}
\MakeOuterQuote{"}

\numberwithin{equation}{section}

\newtheorem{lem}{Lemma}[section]
\newtheorem{prop}{Proposition}[section]

\newtheorem{defn}{Definition}[section]

\newtheorem{conj}{Conjecture}[section]

\newtheorem{rem}{Remark}[section]

\usepackage{thmtools,thm-restate}

\DeclareMathOperator{\C}{\mathbb{C}}
\DeclareMathOperator{\R}{\mathbb{R}}
\DeclareMathOperator{\Q}{\mathbb{Q}}
\DeclareMathOperator{\Z}{\mathbb{Z}}
\DeclareMathOperator{\N}{\mathbb{N}}
\DeclareMathOperator{\F}{\mathbb{F}}
\DeclareMathOperator{\FF}{\mathcal{F}}
\DeclareMathOperator{\Av}{Av}
\DeclareMathOperator{\sgn}{sgn}
\DeclareMathOperator{\rank}{rank}
\DeclareMathOperator{\e}{e}

\newcommand{\Mod}[1]{\ (\text{mod}\ #1)}

\title{Computing the average root number of an elliptic surface}
\author{Jake Chinis\vspace{-1.5\baselineskip}}

\begin{document}

\maketitle

\begin{abstract}
By considering a one-parameter family of elliptic curves defined over $\Q$, we might ask ourselves if there is any bias in the distribution (or parity) of the root numbers at each specialization. From the work of Helfgott, we know (at least conjecturally) that the average root number of an elliptic curve defined over $\Q(T)$ is zero as soon as there is a place of multiplicative reduction over $\Q(T)$ other than $-\deg$. Recently, Helfgott's work was extended by Desjardins, where she relaxes some of Helfgott's hypotheses and is able to provide unconditional results on the variation of the root number for many elliptic surfaces. 

In this paper, we are concerned with elliptic curves defined over $\Q(T)$ with no place of multiplicative reduction over $\Q(T)$, except possibly at $-\deg$. More precisely, we will use the work of Helfgott to compute the average root number of an explicit family of elliptic curves defined over $\Q$ and show that this family is "parity-biased" infinitely-often. 
\end{abstract}

\input{Introduction}
\input{Computations}

\section{Acknowledgements}
The author would like to thank Chantal David, Hershy Kisilevsky, and Patrick Meisner, for their thoughts and comments. The author would also like to extend his gratitude to Christophe Delaunay, for verifying the computations contained herein with the use of the PARI/GP software \cite{Par}.

\bibliographystyle{alpha}
\bibliography{main}

\end{document}

%% file: Introduction.tex
\section{Introduction}

Let $E$ be an elliptic curve defined over $\Q$. For every prime $p,$ let $\tilde{E}_p$ denote the reduction of $E$ modulo $p$ and set $a_p:=p+1-\#\tilde{E}_p(\F_p)$, where $\#\tilde{E}_p(\F_p)$ denotes the number of $\F_p$-points on $\tilde{E}_p$. The $L$-series associated to $E$ is defined by the Euler product
\begin{align*}
    L(s,E):=\prod_{\substack{{p \text{ prime}}\\{p|\Delta}}}(1-a_pp^{-s})^{-1} \prod_{\substack{{p \text{ prime}}\\{p\nmid\Delta}}}(1-a_pp^{-s}+p^{1-2s})^{-1},
\end{align*}
where $\Delta$ is the discriminant of $E$. It is well known that the product defining $L(s,E)$ converges and gives rise to an analytic function, provided $\Re(s)>\frac{3}{2}.$ The Modularity Theorem \cite{Wiles} tells us that much more is true; namely,
\begin{align*}
    \Lambda(s,E):=N_E^{\frac{s}{2}}(2\pi)^{-s}\Gamma(s)L(s,E),
\end{align*}
has an analytic continuation to the entire complex plane and satisfies the functional equation
\begin{align*}
 \Lambda(s,E)=w\Lambda(2-s,E),   
\end{align*}
for some $w=w_E=\pm 1$, where $N_E=N_{E/\Q}$ is the conductor of $E$ and where $\Gamma(s):=\int_{0}^\infty t^{s-1}\e^{-t}dt$ is the Gamma function. We call $w$ the \textit{root number of $E$}. 

In this paper, we use the techniques developed by Rizzo \cite{Riz} and generalized by Helfgott \cite{Hel09} to compute the average root number of an explicit family of elliptic curves defined over $\Q$. By a family of elliptic curves defined over $\Q$, we mean an elliptic curve defined over $\Q(T)$; equivalently, it is a one-parameter family of elliptic curves given by a Weierstrass equation of the form
\begin{align*}
    \FF: y^2=x^3+a_2(T)x^2+a_4(T)x+a_6(T),
\end{align*}
for some $a_2(T),a_4(T),a_6(T)\in \Z[T]$. For every $t\in\Z$, we let $\FF(t)$ denote the specialization of $\FF$ at $t$ and note that $\FF(t)$ defines an elliptic curve for all but finitely-many $t$. Moreover, the map which sends $\FF$ to $\FF(t)$ is injective for all but finitely-many $t$ (Silverman's Specialization Theorem, \cite{Spec}). From here, we let
\begin{align*}
    \varepsilon_{\FF}(t):=
      \left\{\def\arraystretch{1.2}%
\begin{array}{@{}l@{\quad}l@{}}
\mbox{the root number of $\FF(t)$} & \mbox{if $\FF(t)$ is an elliptic curve,}\\
0 & \mbox{otherwise,}
\end{array}\right.
\end{align*}
and define the \textit{average root number of $\FF$ over $\Z$} by
\begin{align*}
    \Av_{\Z}(\varepsilon_{\FF}):=\lim_{T\rightarrow \infty}\frac{1}{2T} \sum_{|t|\leq T}\varepsilon_{\FF}(t),
\end{align*}
provided the limit exists.

In \cite{Hel09}, Helfgott showed (conditionally, and unconditionally in some cases) that $\Av_{\Z}(\varepsilon_{\FF})=0$ whenever $\FF$ has a place of multiplicative reduction over $\Q(T)$ other than $-\deg$. In order to make the statement precise, we first state the following conjectures:

\begin{conj}[Chowla's Conjecture]
Let $P$ be a squarefree polynomial with integer coefficients. Then,
\begin{align*}
    \lim_{N\rightarrow \infty} \frac{1}{N}\sum_{n\leq N} \lambda(P(n))=0,
\end{align*}
where $\lambda(n):=\prod_{p|n}(-1)^{\nu_p(n)}$ is Liouville's function and where $\nu_p(n)$ denotes the $p$-adic valuation of $n$.
\end{conj}

\begin{rem}
By "Strong Chowla's Conjecture" for a polynomial $P$, we mean that Chowla's Conjecture holds for $P(ax+b)$ for all $a,b\in \Z, a\neq 0$.
\end{rem}

\begin{conj}[Squarefree Sieve Conjecture]
\label{SFSC}
Let $P$ be a squarefree polynomial with integer coefficients. Then,
\begin{align*}
    \lim_{N\rightarrow \infty}\frac{1}{N}\#\{1\leq n \leq N: \exists \text{ prime $p>\sqrt {N}$ s.t. $p^2|P(n)$}\}=0.
\end{align*}
\end{conj}

\begin{prop}[\cite{Hel09}]
\label{ppb}
Let $\FF$ be a family of elliptic curves defined over $\Q$. Let $M_{\FF}(T)$ and $B_{\FF}(T)$ be the polynomials defined by
\begin{align*}
    M_{\FF}(T):=\prod_{\substack{{\nu \text{ mult.}}\\{\nu\neq -\deg}}}Q_{\nu}(T), \qquad
    B_{\FF}(T):=\prod_{\substack{{\nu \text{ quite bad}}\\{\nu\neq -\deg}}}Q_{\nu}(T),
\end{align*}
where the products are over all places $\nu$ of $\Q(T)$ for which $\FF$ has multiplicative reduction over $\Q(T)$ and quite bad \footnote{$\nu$ is a place of quite bad reduction if no quadratic twist of $\FF$ has good reduction at $\nu$.} reduction over $\Q(T)$, respectively, and where $Q_\nu(T)$ is the polynomial associated to $\nu$. Then, for all but finitely-many $t\in \Z$,
\begin{align*}
    \varepsilon_{\FF}(t)=\sgn(g_\infty(t))\lambda(M_{\FF}(t))\prod_{\text{$p$ prime}} g_p(t),
\end{align*}
where $g_\infty$ is a polynomial, $\sgn(g_\infty(t))$ denotes the sign of $g_\infty$ at $t$, and $g_p:\Q_p\rightarrow \{\pm 1\}$ are functions satisfying:
\begin{itemize}
    \item $g_p$ are locally constant outside a finite set of points;
    \item for all but finitely-many primes $p$, $g_p(t)=1$ whenever $\nu_p(B_{\FF}(t))<2$.
\end{itemize}
Moreover, if $\FF$ has at least one place of multiplicative reduction over $\Q(T)$ other than $-\deg$, and if the Squarefree Sieve Conjecture holds for $B_{\FF}(T)$ and Strong Chowla's Conjecture holds for $M_{\FF}(T)$, then $\Av_{\Z}(\varepsilon_{\FF})=0$. On the other hand, if $\FF$ has no place of multiplicative reduction over $\Q(T)$, except possibly at $-\deg$, and if the Squarefree Sieve Conjecture holds for $B_{\FF}(T),$ then
\begin{align*}
    \Av_{\Z}(\varepsilon_{\FF})=\frac{c_-+c_+}{2}\prod_{p \text{ prime}}\int_{\Z_p}g_p(t)dt,
\end{align*}
where $dt$ denotes the usual $p$-adic measure and where $c_{\pm}=\lim_{x\rightarrow \pm \infty}\sgn(g_\infty(x))$.
\end{prop}

\begin{rem}
The above theorem is conditional on the Squarefree Sieve Conjecture as well as on Chowla's Conjecture, which are known to hold in some cases; namely, Chowla's Conjecture is known to hold for polynomials of degree 1, whereas the Squarefree Sieve Conjecture is known to hold for polynomials whose irreducible factors have degree less than or equal to 3 \cite{Hel04}. In \cite{Des}, Desjardins improves upon the work of Helfgott by replacing the Squarefree Sieve Conjecture with some technical hypotheses, thus providing unconditional results on the distribution of root numbers for many families of elliptic curves.
\end{rem}

There has been little work dealing with the case where $\FF$ has no place of multiplicative reduction over $\Q(T)$, except possibly at $-\deg$. In \cite{Riz}, Rizzo showed that Washington's family \cite{Was} $\mathcal{W}: y^2=x^3+tx^2-(t+3)x+1$  has $\varepsilon_{\mathcal{W}}(t)=-1$ for all $t\in \Z$ (so that $\Av_{\Z}(\varepsilon_{\mathcal{W}})$ is trivially non-zero) and he also gave an example of a family of elliptic curves whose $j$-invariant is not constant and whose average root number is not equal to $\pm 1, 0$. There are other such isolated examples, see \cite{BDD} for a more thorough survey. 

In \cite{BDD}, the authors present a systematic approach to describing families of elliptic curves defined over $\Q$ whose average root number is not zero: they classify all such "potentially parity-biased" families whose factors, in the parameter $T$, have degree less than or equal to $2$. More precisely,

\begin{defn}
Let $\FF$ be an elliptic curve defined over $\Q(T)$, let $j_{\FF}(T)$ denote the $j$-invariant of $\FF$, and let $r_{\FF}$ denote the rank of $\FF$ over $\Q(T)$. Then,
\begin{itemize}
\item $\FF$ is \textbf{potentially-parity biased over $\Z$} if $\FF$ has no place of multiplicative reduction over $\Q(T)$, except possibly at $-\deg$;
\item $\FF$ is \textbf{parity-biased over $\Z$} if $\Av_{\Z}(\varepsilon_{\FF})$ exists and is non-zero;
\item $\FF$ is \textbf{non-isotrivial} if $j_{\FF}(T)$ is non-constant; otherwise, $\FF$ is \textbf{isotrivial};
\item $\FF$ has \textbf{excess rank} if  $\Av_{\Z}(\varepsilon_{\FF})$ exists and $\Av_{\Z}(\varepsilon_{\FF})=-(-1)^{r_{\FF}}$.
\end{itemize}
\end{defn}

\begin{rem}
As the authors in \cite{BDD} remark, there are many examples of isotrivial families. For example, quadratic twists of a fixed elliptic curve $E:y^2=x^3+a_2x^2+a_4x+a_6$ defined over $\Q$ by a polynomial $d(T)\in\Z[T]$, $E^{d(T)}:d(T)y^2=x^3+a_2x^2+a_4x+a_6$, $a_i\in \Z,i=2,4,6$. Furthermore, we have the following implications:
\begin{align*}
    \mbox{Excess Rank}\Rightarrow \mbox{Parity-Biased} \xRightarrow[Conj.]{Helfgott} \mbox{Potentially Parity-Biased.}
\end{align*}
\end{rem}

In Theorems 7 and 8 of \cite{BDD}, the authors show that there are essentially $6$ different classes of non-isotrivial, potentially parity-biased families of elliptic curves defined over $\Q$ whose coefficients, in the parameter $T$, have degree less than or equal to $2$; namely,
\begin{align*}
    &\FF_s(t):y^2=x^3+3tx^2+3sx+st, \mbox{ with $s\in \Z_{\neq 0}$};\\
    &\mathcal{G}_w(t):wy^2=x^3+3tx^2+3tx+t^2, \mbox{ with $w\in \Z_{\neq 0}$};\\
    &\mathcal{H}_w(t):wy^2=x^3+(8t^2-7t+3)x^2-3(2t-1)x+(t+1), \mbox{ with $w\in \Z_{\neq 0}$};\\
    &\mathcal{I}_w(t):wy^2=x^3+t(t-7)x^2-6t(t-6)x+2t(5t-27),\mbox{ with $w\in \Z_{\neq 0}$};\\
    &\mathcal{J}_{m,w}(t):wy^2=x^3+3t^2x^2-3mtx+m^2, \mbox{ with $m,w\in \Z_{\neq 0}$};\\
    &\mathcal{L}_{w,s,v}(t): wy^2=x^3+3(t^2+v)x^2+3sx+s(t^2+v), \mbox{ with $v\in \Z,s,w\in \Z_{\neq 0}$}.
\end{align*}
The authors then compute the average root number for two subfamilies of $\FF_s$,
\begin{align*}
    &\mathcal{W}_a(t):y^2=x^3+tx^2-a(t+3a)x+a^3, \mbox{ with $a\in \Z_{\neq 0}$},\\
    &\mathcal{V}_a(t):y^2=x^3+3tx^2+3atx +a^2t, \mbox{ with $a\in \Z_{\neq 0}$},
\end{align*}
highlighting the key ideas in implementing Helfgott's and Rizzo's work (see also the "Sketch of the proof of Theorem 6" on pages 6-9 in \cite{BDD}, where the authors give a general overview on the correct way to proceed). 
\begin{rem}
Note that $\mathcal{W}_a(t)\cong \FF_{-3^54a^2}(12t+18a)$ and $\mathcal{V}_a(t)\cong \FF_{4a^2}(4t-2a)$.
\end{rem}

In this paper, we complement the work of \cite{BDD} by computing $\Av_{\Z}(\varepsilon_{\FF_s})$; that is, we prove the following:

\begin{restatable}{thm}{mainthm}
\label{mainthm}
Let $\FF_s$ denote the family of elliptic curves defined over $\Q$ whose specializations are given by the Weierstrass equation
\begin{align*}
    \FF_s(t):y^2=x^3+3tx^2+3sx+st, \mbox{ with $s\in \Z_{\neq 0}$}.
\end{align*}
Then, $\Av_{\Z}(\varepsilon_{\FF_s})$ exists with
\begin{align*}
    \Av_{\Z}(\varepsilon_{\FF_s})=-\prod_{\text{$p$ prime}} E_{\FF_s}(p),
\end{align*}
where the $E_{\FF_s}(p)$ are given by Propositions \ref{p>5}, \ref{p=3}, and \ref{p=2}, for $p\geq 5$, $p=3$, and $p=2$, respectively. In particular, $\FF_s$ is parity biased over $\Z$ iff $s\not\equiv 1,3,5 \Mod{8}$.
\end{restatable}

\subsection{Applications}

In this section, we present some areas of mathematics where average root numbers play a role. We only briefly discuss the results here, leaving the rest to the imagination.

\subsubsection{One-level density functions of elliptic surfaces}

As mentioned in \cite{BDD}, the average root numbers of elliptic surfaces defined over $\Q$ appear naturally in the study of elliptic curves and their associated $L$-functions. They show in upcoming work that the one-level density function of an elliptic surface $\FF$, denoted by $W_{\FF}$, is equal to
\begin{align*}
    W_{\FF}(\tau)=r_{\FF}\delta_0(\tau)+\frac{1+(-1)^{r_{\FF}}\Av_{\Z}(\varepsilon_{\FF})}{2}W_{\text{SO(even)}}(\tau)+\frac{1-(-1)^{r_{\FF}}\Av_{\Z}(\varepsilon_{\FF})}{2}W_{\text{SO(odd)}}(\tau),
\end{align*}
where $r_{\FF}$ is the rank of $\FF$ over $\Q(T)$, $\delta_0$ is the Dirac measure at $0$, and $W_{\text{SO(even)}}$ (resp. $W_{\text{SO(odd)}}$) is the one-level density function of the special orthogonal group of even size (resp. odd size). For more on one-level densities and applications of Helfgott's work, see \cite{Mil}.

\subsubsection{Constructing families of elliptic curves with elevated rank}

Assuming the Birch-Swinnerton-Dyer Conjecture, Silverman's Specialization Theorem \cite{Spec} tells us that 
\begin{align*}
    \rank(\FF(t)(\Q))\geq r_{\FF}+\frac{1}{2}(1-\varepsilon_{\FF}(t)(-1)^{r_{\FF}})
\end{align*}
for all but finitely-many $t\in\Q$; in particular, the average root number of $\FF$ provides a lower bound for the rank of each specialization. In \cite{CCH}, the authors use this lower bound to construct families of elliptic curves with \textit{elevated rank}; that is, to construct families of elliptic curves for which $r_{\FF}$ is strictly less than $\rank(\FF(t)(\Q))$ for all but finitely-many $t$.

\begin{rem}
Without assuming BSD, Silverman's Specialization Theorem tells us that $r_{\FF}\leq \rank(\FF(t)(\Q))$ for all but finitely-many $t\in\Q$.
\end{rem}

\subsubsection{Generalizing the congruent number problem}

Given an angle $\frac{\pi}{3}\leq \theta \leq \pi$, a squarefree integer $n$ is called \textit{$\theta$-congruent} if there exists a triangle whose largest angle is $\theta$, whose sides are all rational, and whose area is $n.$ In \cite{Rol}, the author gives an elliptic curve criterion for when a given integer is $\theta$-congruent, he then uses the work of Helfgott \cite{Hel09} to prove some density results concerning $\theta$-congruent numbers.

\subsection{Overview of this paper}

In this section, we provide a general overview of the work contained herein. Once again, our goal is to combine the work of Helfgott \cite{Hel09} and Rizzo \cite{Riz} to compute the average root number of an explicit family of elliptic curves defined over $\Q$. The main tool in proving Theorem \ref{mainthm} is the work of Helfgott; namely,
\begin{prop}[\cite{Hel09},Proposition 7.7]
\label{main_tool}
Let $S$ be a finite set of places of $\Q$, including the place at infinity. For every place $\nu \in S$, let $g_\nu:\Q_\nu\rightarrow \C$ be a bounded function that is locally constant almost everywhere. For every prime $p\not \in S$, let $h_p:\Q_p\rightarrow \C$ be a function that is locally constant almost everywhere and such that $|h_p(x)|\leq 1$ for all $x$. Let $B(x)\in \Z[x]$ be a non-zero polynomial and assume that $h_p(x)=1$ whenever $\nu_p(B(x))<2$. Let 
\begin{align*}
W(n)=\prod_{\nu\in S}g_\nu(n)\prod_{p\not \in S}h_p(n).    
\end{align*}
If the Squarefree Sieve Conjecture holds for $B(x)$, then
\begin{align*}
    \Av_{\Z}(W)=\frac{c_{-}+c_{+}}{2}\prod_{p\in S}\int_{\Z_p}g_p(x)dx\prod_{p\not\in S}\int_{\Z_p}h_p(x)dx,
\end{align*}
where $c_\pm=\lim_{x\rightarrow \pm \infty}g_\infty(x)$ and where $\Av_{\Z}(W):=\lim_{N\rightarrow \infty} \frac{1}{2N}\sum_{|n|\leq N}W(n)$.
\end{prop}

\begin{rem}
(i) When we say that a function is locally constant almost everywhere, we mean that it is locally constant outside a finite set of points. Recall further that a function $f$ from topological space $X$ into a set $Y$ is \textit{locally constant} if for every $x\in X$ there exists a neighbourhood $U$ about $x$ such that $f$ is constant on $U$.

(ii) We use $\nu$ to represent a place of $\Q$ that is either finite or infinite, so that $\Q_{\nu}=\Q_p$ is the field of $p$-adic numbers if $\nu=p$ is a (finite) prime and $\Q_{\nu}=\R$ if $\nu=\infty$ is the prime/place at infinity. The products indexed by $p$ are over finite primes, under the respective conditions.

(iii) Note that a function $f:\R \rightarrow \C$ that is locally constant almost everywhere (that is, outside a finite set of points) is a step function with finitely-many discontinuities; in particular, $g_\infty(x)$ is constant for all $x$ sufficiently large (sufficiently large and negative, respectively).
\end{rem}

In order to use Proposition \ref{main_tool}, our first goal is to write $\varepsilon_{\FF}(t)$ as an infinite product: this is accomplished by writing the root number of $\FF(t)$ as a product of \textit{local root numbers $w_p(t)$}, 
\begin{align*}
\varepsilon_{\FF}(t)=-\prod_{\text{$p$ prime}} w_p(t).
\end{align*}
Alternatively, one may define the root number $w$ of an elliptic curve $E/\Q$ to be the infinite product of local root numbers (independently of the functional equation associated to $L(s,E)$). The local root numbers are themselves defined by representations of the Weil-Deligne group of $\Q_p$ (with $w_\infty=-1$ for all elliptic curves defined over $\R$); we refer the reader to \cite{del} and \cite{Tate}. 

\begin{rem}
The local root numbers $w_p(t)$ are essentially determined by the reduction type of $\FF(t)$ at $p$; see the tables of Rohrlich \cite{Roh}, Connell \cite{Con}, and Halberstadt \cite{Hal}. Note that, in \cite{Hal}, Halberstadt requires a minimal Weierstrass equation for $\FF(t)$. In \cite{Riz}, Rizzo removes the minimality conditions on $\FF(t)$. For this reason, we use Tables I, II, and III in \cite{Riz} to compute our local root numbers\footnote{As the authors in \cite{BDD} remark, there are the following misprints in \cite{Riz}: in Table II, the line corresponding to $(a,b,c)=(\geq 5,6,9)$ should read $c_6^\prime+2\not\equiv3c_{4,4}\Mod{9}$; in Table III, the second line should read $(a,b,c)=(0,0,\geq 0)$ and the Kodaira symbol at $(a,b,c)=(2,3,1)$ should read $I_2^*$.}.
\end{rem}

Sadly, the local root numbers do not, in general, satisfy the hypotheses of Proposition \ref{main_tool} (see section 1.2 of \cite{Hel09}). In order to rectify this, we then express $\varepsilon_{\FF}(t)$ as a product of modified local root numbers $w_\nu^*(t)$, 
\begin{align*}
\varepsilon_{\FF}(t)=-w_\infty^*(t)\prod_{\text{$p$ prime}} w_p^*(t),
\end{align*}
with $w_\nu^*(t)$ satisfying the hypotheses of Proposition \ref{main_tool}; our choice of $w_\nu^*(t)$ is a natural one (see Remark \ref{mod_loc_rem}). At this point, computing the average root number of $\FF$ amounts to computing the $p$-adic integrals $\int_{\Z_p}w_p^*(t)dt$, which we break into three sections (for $p\geq 5, p=3, p=2$), and we have that
\begin{align*}
    \Av_{\Z}(\varepsilon_{\FF})=-\prod_{\text{$p$ prime}}\int_{\Z_p}w_p^*(t)dt,
\end{align*}
as our choice of $w_\infty^*(t)$ is equal to $1$ for all but finitely-many $t\in\Z$.

\begin{rem}
In all that follows, the letter $p$ will denote a (finite) prime and products over $p$ are understood to be over all (finite) primes. In the case where a product involves the added "prime/place at infinity," we will make this explicit by writing the product over $p\leq \infty$. As usual, $\Z_p$ denotes the ring of $p$-adic integers and for all $n\in \Z_p$, $\nu_p(n)$ denotes the $p$-adic valuation of $n$. We use the identification $\Z\hookrightarrow \Z_p$ freely and set $n_p:=np^{-\nu_p(n)}$ for all $n\in \Z_p\setminus \{0\}.$
\end{rem}

%% file: Computations.tex
\section{The family $\FF_s$ and its average root number}

From now on, we concern ourselves with the Weierstrass equation
\begin{align*}
    \FF_s(t):y^2=x^3+3tx^2+3sx+st, \mbox{ $s\in \Z, s\neq 0$},
\end{align*}
for which we have
\begin{align*}
    &c_4(t)=2^43^2(t^2-s),\\
    &c_6(t)=-2^63^3t(t^2-s),\\
    &\Delta(t)=-2^63^3s(t^2-s)^2,\\
    &j(t)=\frac{-2^63^3}{s}(t^2-s).
\end{align*}
We prove the following:

\mainthm*

\section{Modifying the local root numbers of $\FF_s(t)$}

The local root numbers of $\FF_s(t)$ can be found in Appendix A of \cite{BDD}. For convenience to the reader, we list the results for $p\geq 5$:

\begin{prop}
\label{local_root_p>5}
For $p \geq 5$,
\begin{itemize}
    \item if $0\leq 2\nu_p(t)<\nu_p(s)$, then
    \begin{align*}
    w_p(t)=
   \left\{\def\arraystretch{1.2}%
\begin{array}{@{}l@{\quad}l@{}}
-\Big(\frac{3t_p}{p}\Big) & \mbox{if $\nu_p(t)$ is even,}\\
\Big(\frac{-1}{p}\Big) & \mbox{if $\nu_p(t)$ is odd;}
\end{array}\right.
\end{align*}
\item if  $0\leq \nu_p(s)<2\nu_p(t)$, then
  \begin{align*}
    w_p(t)=
   \left\{\def\arraystretch{1.2}%
\begin{array}{@{}l@{\quad}l@{}}
\Big(\frac{-1}{p}\Big)^{\frac{\nu_p(s)}{2}} & \mbox{if $\nu_p(s)$ is even,}\\
\Big(\frac{-2}{p}\Big)& \mbox{if $\nu_p(s)$ is odd;}
\end{array}\right.
\end{align*}
\item if $0\leq 2\nu_p(t)=\nu_p(s)$, then
\begin{align*}
    w_p(t)=
   \left\{\def\arraystretch{1.2}%
\begin{array}{@{}l@{\quad}l@{}}
\Big(\frac{-1}{p}\Big) & \mbox{if $\nu_p(t)+\nu_p(t^2-s)\equiv 1 \Mod{2}$,}\\
\Big(\frac{-3}{p}\Big)& \mbox{if $\nu_p(t)+\nu_p(t^2-s)\equiv 2,4 \Mod{6}$,}\\
1 & \mbox{if $\nu_p(t)+\nu_p(t^2-s)\equiv 0\Mod{6}$.}
\end{array}\right.
\end{align*}
\end{itemize}
\end{prop}

\begin{rem}
$\Big(\frac{\cdot}{p}\Big)$ represents the \textit{Legendre symbol}.
\end{rem}

So far, we can write the root number of $\FF_s(t)$ as a product of local root numbers
\begin{align*}
    \varepsilon_{\FF_{s}}(t)=-\prod_{\text{$p$ prime}} w_p(t),
\end{align*}
with $w_p(t)$ given by Proposition \ref{local_root_p>5} for $p\geq 5$ and $w_3(t),w_2(t)$ as in Appendix A of \cite{BDD}. Our next goal is to modify the local root numbers in order to apply Proposition \ref{main_tool}.

\begin{lem}
\label{w_p^*}
For $p\geq 5$, let $w_p^*(t)=w_p(t)\Big(\frac{-1}{p}\Big)^{\nu_p(t^2-s)}$. For $p=2,3,$ and for the prime at infinity, let $w_2^*(t),w_3^*(t),w_\infty^*(t)\in \{\pm 1\}$ be defined by
\begin{align*}
    &w_2^*(t)\equiv (t^2-s)_2w_2(t) \Mod{4},\\
    &w_3^*(t)=(-1)^{\nu_3(t^2-s)}w_3(t),\\
    &w_\infty^*(t)=\sgn(t^2-s).
\end{align*}
Then,
\begin{align}
\label{prod1=prod2}
    \varepsilon_{\FF_{s}}(t)=-\prod_{\text{$p$ prime}} w_p(t)=-w_\infty^*(t)\prod_{\text{$p$ prime}} w_p^*(t).
\end{align}
\end{lem}

\begin{rem}
\label{mod_loc_rem}
The choice of $w_p^*(t)$ is a natural one, more or less. We begin by assuming $p\geq 5$, $p\nmid s$, and $p| \Delta(t)=-2^63^3s(t^2-s)^2$, so that $\nu_p(t^2-s)>0$ (if $p\nmid \Delta,$ then $w_p(t)=1$ and this does not pose a problem in applying Proposition \ref{main_tool}; similarly, the assumption that $p\nmid 6s$ throws away a finite number of primes, which will belong to the set $S$ in Proposition \ref{main_tool}). We have two cases to consider: $\nu_p(t)=\nu_p(s)=0$ and $\nu_p(t)>\nu_p(s)=0$. In the first case,
\begin{align*}
     w_p(t)=
   \left\{\def\arraystretch{1.2}%
\begin{array}{@{}l@{\quad}l@{}}
\Big(\frac{-1}{p}\Big) & \mbox{if $\nu_p(t^2-s)\equiv 1\Mod{2}$,}\\
\Big(\frac{-3}{p}\Big) & \mbox{if $\nu_p(t^2-s)\equiv 2,4 \Mod{6},$}\\
1 & \mbox{if $\nu_p(t^2-s)\equiv 0\Mod{6},$}
\end{array}\right.
\end{align*}
whereas, $w_p(t)=1$ in the second case. Taking
\begin{align*}
    w_p^*(t)=w_p(t)\Big(\frac{-1}{p}\Big)^{\nu_p(t^2-s)},
\end{align*}
we see that $w_p^*(t)=1$ whenever $p\nmid 6s$ and $\nu_p(t^2-s)\leq 1$. The choices of $w_2^*(t),w_3^*(t),w_\infty^*(t)$ are then made so that Equation \ref{prod1=prod2} holds. Combining this remark together with Lemma \ref{w_p^*} allows us to apply Proposition \ref{main_tool}.
\end{rem}

\begin{proof}
For $p$ odd, $\Big(\frac{-1}{p}\Big)\equiv p\Mod{4}$, so that
\begin{align*}
    \prod_{p\neq 2,3}w_p^*(t)
    &=\prod_{p\neq 2,3}\Big(\frac{-1}{p}\Big)^{\nu_p(t^2-s)}\prod_{p\neq 2,3} w_p(t)\\
    &\equiv \prod_{p\neq 2,3} p^{\nu_p(t^2-s)}\prod_{p\neq 2,3} w_p(t) \Mod{4}\\
    &\equiv (-1)^{\nu_3(t^2-s)}\prod_{p\neq 2}p^{\nu_p(t^2-s)}\prod_{p\neq 2,3} w_p(t) \Mod{4}\\
    &=(-1)^{\nu_3(t^2-s)}|(t^2-s)_2|\prod_{p\neq 2,3} w_p(t);
\end{align*}
thus,
\begin{align*}
    -w_\infty^*(t)\prod_{\text{$p$ prime}} w_p^*(t)=-\prod_{\text{$p$ prime}} w_p(t).
\end{align*}
\end{proof}

Applying Proposition \ref{main_tool} with $S=\{p:p\nmid 6s\}\cup \{\infty\},g_\nu=w_\nu^*,h_p=w_p^*,$ and $B(x)=x^2-s$, we have that
\begin{align*}
    \Av_{\Z}(\varepsilon_{\FF_{s}})=-\prod_{\text{$p$ prime}} \int_{\Z_p}w_p^*(t)dt,
\end{align*}
as $w_\infty^*(t)=1$ for all but finitely-many integers $t$. 

\begin{rem}
Recall that the Squarefree Sieve Conjecture (Conjecture \ref{SFSC}) holds for all polynomials whose irreducible factors are of degree $\leq 3$ \cite{Hel04}. Since we are applying Proposition \ref{main_tool} with $B(x)=x^2-s$, our results are unconditional.
\end{rem}

The next few sections are devoted to computing the $p$-adic integrals $\int_{\Z_p}w_p^*(t)dt$ for $p\geq 5,p=3$, and $p=2$, respectively.

\section{$p$-uniformly locally constant multiplicative functions}

In our work, we deal with functions Rizzo calls \textit{$p$-uniformly locally constant multiplicative functions}. We will see that these functions are locally constant everywhere, except possibly at $0$, which is what we need in order to apply Proposition \ref{main_tool}.

\begin{defn} [\cite{Riz}, p.11]
A function $f:\Z_p\rightarrow \R$ is a \textbf{$p$-uniformly locally constant multiplicative function} if there exists a positive integer $\eta$ such that the value of $f$ at $x\in \Z_p$ is completely determined by $\nu_p(x)$ and $x_p:=xp^{-\nu_p(x)}\Mod{p^\eta}$. We call $\eta$ a \textbf{uniformity constant} of $f$
\end{defn}

\begin{rem}
Note that uniformity constants are not unique: if the value of $f$ at $x$ is determined by $\nu_p(x)$ and $x_p \Mod{p^\eta}$, then it is certainly determined by $\nu_p(x)$ and $x_p\Mod{p^{\eta^\prime}}$ for any $\eta^\prime\geq \eta$.
\end{rem}

From the definition above, it should be clear that all $p$-uniformly locally constant multiplicative functions are locally constant on $p^e\Z_p^*:=\{x\in \Z_p:\nu_p(x)=e\}$ for all $e\geq 0$. To see this, let $\eta$ be a uniformity constant of $f$, partition $p^e\Z_p^*$ into $p^{\eta-1}(p-1)$ disjoint balls of radius $p^{e+\eta}$,
\begin{align*}
    p^e\Z_p^*=\bigcup_{\substack{{\alpha_i=0,1,\dots,p-1}\\{\alpha_0\neq 0}}} p^e(\alpha_0+\alpha_1 p+\dots + \alpha_{\eta-1}p^{\eta-1})+p^{e+\eta}\Z_p,
\end{align*}
and note that $f$ is constant on each ball. From here, it is easy to see that
\begin{align*}
    \int_{\nu_p(t)=e} f(t)dt:=\int_{p^e\Z_p^*}f(t)dt=\sum_{d\in (\Z/p^\eta\Z)^*}\frac{f(dp^e)}{p^{e+\eta}}.
\end{align*}
We extend the above expression to all of $\Z_p$ by writing
\begin{align*}
    \int_{\Z_p}f(t)dt=\sum_{e=0}^\infty \int_{\nu_p(t)=e}f(t)dt,
\end{align*}
provided the sum converges absolutely.

\section{Computing $\int_{\Z_p}w_p^*(t)dt$ for $p\geq 5$}

During the calculations involved in computing $\int_{\Z_p}w_p^*(t)dt$ for $p\geq 5$, we will need to deal with integrals of the form
\begin{align*}
    \int_{\substack{{\nu_p(t)=\frac{\nu_p(s)}{2}}\\{\nu_p(t^2-s)=\nu_p(s)+k}}} 1 dt,
\end{align*}
for $k\in\Z_{\geq 0}$; this is accomplished in the following lemma,

\begin{lem}
\label{lem_meas}
For $k\in\Z_{\geq 0}$, let $S_k:=\{t\in\Z_p:\nu_p(t)=\frac{\nu_p(s)}{2},\nu_p(t^2-s)=\nu_p(s)+k\}$. Then, $S_k$ has measure
\begin{align*}
   \mu(S_k)= \left\{\def\arraystretch{1.2}%
\begin{array}{@{}l@{\quad}l@{}}
0 & \mbox{if $\nu_p(s)$ is odd,}\\
 \left\{\def\arraystretch{1.2}%
\begin{array}{@{}l@{\quad}l@{}}
\frac{p-1}{p^{\frac{\nu_p(s)}{2}+1}}& \mbox{if $\Big(\frac{s_p}{p}\Big)=-1$ and $k=0$,}\\
\frac{p-3}{p^{\frac{\nu_p(s)}{2}+1}}& \mbox{if $\Big(\frac{s_p}{p}\Big)=1$ and $k=0,$}\\
0 & \mbox{if $\Big(\frac{s_p}{p}\Big)=-1$ and $k\geq 1$,}\\
\frac{2(p-1)}{p^{\frac{\nu_p(s)}{2}+k+1}} & \mbox{if $\Big(\frac{s_p}{p}\Big)=1$ and $k\geq 1$,}
\end{array}\right.& \mbox{if $\nu_p(s)$ is even.}
\end{array}\right.
\end{align*}
\end{lem}

\begin{proof}
We assume that $\nu_p(s)$ is even; otherwise, $S_k=\emptyset$ and there is nothing to prove. Let $\chi_k$ denote the characteristic function of $S_k$. For $t\in \Z_p$, $\chi_k(t)=1$ iff $\nu_p(t)=\frac{\nu_p(s)}{2}$ and $t_p^2\in s_p+p^k\Z_p^*$. Point is, $\chi_k$ is a $p$-uniformly locally constant multiplicative function with uniformity constant $\eta=k+1$. Hence,
\begin{align*}
    \mu(S_k)&:=\int_{\substack{{\nu_p(t)=\frac{\nu_p(s)}{2}}\\{\nu_p(t^2-s)=\nu_p(s)+k}}} 1 dt\\
    &=\frac{1}{p^{\frac{\nu_p(s)}{2}+k+1}}\sum_{d\in(\Z/p^{k+1}\Z)^*}\chi_k(dp^\frac{\nu_p(s)}{2}).
\end{align*}
We begin with the case $k=0$ and treat the other cases separately. 

For $k=0$, $\chi_0(dp^\frac{\nu_p(s)}{2})=1$ iff $d^2\not\equiv s_p \Mod{p}$. If $s_p$ is not a square modulo $p$, then all $d\in (\Z/p\Z)^*$ possess the preceding quality; on the other hand, if $s_p$ is a square modulo $p$, exactly two $d\in(\Z/p\Z)^*$ are such that $d^2\equiv s_p \Mod{p}$. Therefore, 
\begin{align*}
    \mu(S_0)=\left\{\def\arraystretch{1.2}%
\begin{array}{@{}l@{\quad}l@{}}
\frac{p-1}{p^{\frac{\nu_p(s)}{2}+1}}& \mbox{if $\Big(\frac{s_p}{p}\Big)=-1,$}\\
\frac{p-3}{p^{\frac{\nu_p(s)}{2}+1}}& \mbox{if $\Big(\frac{s_p}{p}\Big)=1$.}
\end{array}\right.
\end{align*}

Now, suppose that $k\in\N$ and let $S_k^*:=\{t\in\Z_p:\nu_p(t)=\frac{\nu_p(s)}{2},\nu_p(t^2-s)\geq \nu_p(s)+k\}$. Since $S_k=S_k^*\setminus S_{k+1}^*$, with $\mu(S_{k+1}^*)<\infty$, $\mu(S_k)=\mu(S_k^*)-\mu(S_{k+1}^*)$. Moreover, if we let $\chi_k^*$ denote the characteristic function of $S_k^*$, then $\chi_k^*$ is a $p$-uniformly locally constant multiplicative function with uniformity constant $\eta=k$. Therefore,
\begin{align*}
    \mu(S_k^*)=\frac{1}{p^{\frac{\nu_p(s)}{2}+k}}\sum_{d\in(\Z/p^k\Z)^*}\chi_k^*(dp^\frac{\nu_p(s)}{2}),
\end{align*}
with $\chi_k^*(dp^\frac{\nu_p(s)}{2})=1$ iff $d^2\equiv s_p\Mod{p^k}$. Since an integer $a$ relatively prime to $p$ is a square modulo $p$ iff $a$ is a square modulo $p^n$ for every $n\in\N$, we have that
\begin{align*}
     \mu(S_k^*)
     =\left\{\def\arraystretch{1.2}%
\begin{array}{@{}l@{\quad}l@{}}
0 & \mbox{if $\Big(\frac{s_p}{p}\Big)=-1$,}\\
\frac{2}{p^{\frac{\nu_p(s)}{2}+k}} & \mbox{if $\Big(\frac{s_p}{p}\Big)=1$;}
\end{array}\right.
\end{align*}
and so,
\begin{align*}
    \mu(S_k)
    =\mu(S_k^*)-\mu(S_{k+1}^*)
    =\left\{\def\arraystretch{1.2}%
\begin{array}{@{}l@{\quad}l@{}}
0 & \mbox{if $\Big(\frac{s_p}{p}\Big)=-1$}\\
\frac{2(p-1)}{p^{\frac{\nu_p(s)}{2}+k+1}} & \mbox{if $\Big(\frac{s_p}{p}\Big)=1$,}
\end{array}\right.
\end{align*}
as claimed.
\end{proof}


We are now in a position to prove the following:

\begin{prop}
\label{p>5}
For $p\geq 5$,
\begin{align*}
    &\int_{\Z_p}w_p^*(t)dt\\
    &=
    \left\{\def\arraystretch{1.2}%
\begin{array}{@{}l@{\quad}l@{}}
\Big(\frac{-1}{p}\Big)^\frac{
\nu_p(s)}{2}\frac{1}{p^{\frac{
\nu_p(s)}{2}+1}} & \mbox{if $\nu_p(s)$ is even},\\
\Big(\frac{2}{p}\Big)\frac{1}{p^{\frac{\nu_p(s)+1}{2}}} & \mbox{if $\nu_p(s)$ is odd,}
\end{array}\right.\\
&+\left\{\def\arraystretch{1.2}%
\begin{array}{@{}l@{\quad}l@{}}
0 & \mbox{if $\nu_p(s)=0,1,2$,}\\
\Big(\frac{-1}{p}\Big)\frac{p-1}{p^2} & \mbox{if $\nu_p(s)=3,4,5,6$,}\\
\Big(\frac{-1}{p}\Big)\frac{1}{p+1}\cdot
\left\{\def\arraystretch{1.2}%
\begin{array}{@{}l@{\quad}l@{}}
1-p^{-2\alpha} & \mbox{if $\nu_p(s)\equiv 2\Mod{4}$,}\\
1-p^{-2\alpha-2}& \mbox{otherwise,}
\end{array}\right.
&\mbox{if $\nu_p(s)\geq 7$,}
\end{array}\right.\\
&+\left\{\def\arraystretch{1.2}%
\begin{array}{@{}l@{\quad}l@{}}
0 & \mbox{if $\nu_p(s)$ is odd,}\\
\left\{\def\arraystretch{1.2}%
\begin{array}{@{}l@{\quad}l@{}}
\Big(\frac{-1}{p}\Big)^{\frac{j}{2}}\frac{p-1}{p^{\frac{\nu_p(s)}{2}+1}} & \mbox{if $\Big(\frac{s_p}{p}\Big)=-1$,}\\
\Big(\frac{-1}{p}\Big)^{\frac{j}{2}}\frac{p-1}{p^{\frac{\nu_p(s)}{2}+1}} & \mbox{if $\Big(\frac{s_p}{p}\Big)=1$ and $p\equiv 1 \Mod{3}$,}\\
 \Big(\frac{-1}{p}\Big)^{\frac{j}{2}}\frac{1}{p^{\frac{\nu_p(s)}{2}+1}}\Biggr(
p-(2j+1)-4(-1)^\frac{j}{2}\frac{p^4+\frac{j}{2}p^3+p^2+\frac{j}{2}}{(p+1)(p^4+p^2+1)}\Biggr) & \mbox{if $\Big(\frac{s_p}{p}\Big)=1$ and $p\equiv 2 \Mod{3}$,}
\end{array}\right.
& \mbox{if $\nu_p(s)$ is even,}
\end{array}\right.
\end{align*}
where $\alpha=\lfloor \frac{\nu_p(s)-2}{4} \rfloor$ and $j\in\{0,2\}$ is such that $\nu_p(s)\equiv j \Mod{4}$ (for $\nu_p(s)$ even).
\end{prop}

\begin{rem}
In the case where $\nu_p(s)=0$, such a hideous expression reduces to something quite nice; namely,
\begin{align*}
    \int_{\Z_p}w_p^*(t)dt
    =     \left\{\def\arraystretch{1.2}%
\begin{array}{@{}l@{\quad}l@{}}
1 & \mbox{if $\Big(\frac{s_p}{p}\Big)=-1,$}\\
\left\{\def\arraystretch{1.2}%
\begin{array}{@{}l@{\quad}l@{}}
1 & \mbox{if $p\equiv 1 \Mod{3}$,}\\
1-4\frac{p(p^2+1)}{(p+1)(p^4+p^2+1)} & \mbox{if $p\equiv 2 \Mod{3},$}
\end{array}\right.
& \mbox{if $\Big(\frac{s_p}{p}\Big)=1$}.
\end{array}\right.
\end{align*}
\end{rem}

\begin{proof}

By Proposition \ref{local_root_p>5},
\begin{align*}
    \int_{\Z_p}w_p^*(t)dt
    &=\int_{0\leq 2\nu_p(t)<\nu_p(s)}w_p^*(t)dt+\int_{0\leq \nu_p(s)<2\nu_p(t)}w_p^*(t)dt+\int_{0\leq \nu_p(s)=2\nu_p(t)}w_p^*(t)dt\\
    &=\int_{\substack{{0\leq 2\nu_p(t)<\nu_p(s)}\\{2|\nu_p(t)}}}-\Big(\frac{3t_p}{p}\Big)dt
    +\int_{\substack{{0\leq 2\nu_p(t)<\nu_p(s)}\\{2\nmid\nu_p(t)}}}\Big(\frac{-1}{p}\Big)dt\\
    &\qquad+\left\{\def\arraystretch{1.2}%
\begin{array}{@{}l@{\quad}l@{}}
\int_{0\leq \nu_p(s)<2\nu_p(t)}\Big(\frac{-1}{p}\Big)^\frac{
\nu_p(s)}{2}dt & \mbox{if $\nu_p(s)$ is even,}\\
\int_{0\leq \nu_p(s)<2\nu_p(t)}\Big(\frac{2}{p}\Big)dt & \mbox{if $\nu_p(s)$ is odd,}
\end{array}\right.\\
&\qquad\qquad+\sum_{k=0}^\infty \int_{\substack{{0\leq \nu_p(s)=2\nu_p(t)}\\{\nu_p(t^2-s)=\nu_p(s)+k}}}
w_p^*(t)dt,
\end{align*}
where the infinite sum is simply a partition of $\int_{0\leq \nu_p(s)_=2\nu_p(t)}w_p^*(t)dt$. We consider each line separately, noting that the third line is the most difficult to deal with.

We begin by partitioning the first two integrals as a sum over all $t\in\Z_p$ with $\nu_p(t)=2k$ and $\nu_p(t)=2k+1$, respectively, to obtain
\begin{align*}
    \int_{\substack{{0\leq 2\nu_p(t)<\nu_p(s)}\\{2|\nu_p(t)}}}-\Big(\frac{3t_p}{p}\Big)dt
    &=\sum_{0\leq k < \frac{\nu_p(s)}{4}} \int_{\nu_p(t)=2k} -\Big(\frac{3t_p}{p}\Big)dt\\
    &=\sum_{0\leq k < \frac{\nu_p(s)}{4}} -\Big(\frac{3}{p}\Big)\frac{1}{p^{2k+1}}\sum_{d\in(\Z/p\Z)^*}\Big(\frac{d}{p}\Big)
\end{align*}
and
\begin{align*}
    \int_{\substack{{0\leq 2\nu_p(t)<\nu_p(s)}\\{2\nmid\nu_p(t)}}}\Big(\frac{-1}{p}\Big)dt
    &=\sum_{0\leq k < \frac{\nu_p(s)-2}{4}} \int_{\nu_p(t)=2k+1} \Big(\frac{-1}{p}\Big)dt\\
    &=\sum_{0\leq k < \frac{\nu_p(s)-2}{4}}\Big(\frac{-1}{p}\Big)\mu(\{t\in\Z_p:\nu_p(t)=2k+1\}). 
\end{align*}
In the first case, $\int_{\substack{{0\leq 2\nu_p(t)<\nu_p(s)}\\{2|\nu_p(t)}}}-\Big(\frac{3t_p}{p}\Big)dt=0$: simply note that there are exactly $\frac{p-1}{2}$ squares and $\frac{p-1}{2}$ non-squares modulo $p$; i.e., 
\begin{align*}
    \sum_{d\in(\Z/p\Z)^*}\Big(\frac{d}{p}\Big)
=0
\end{align*}
In the second case, $\mu(\{t\in\Z_p:\nu_p(t)=2k+1\})=\frac{p-1}{p^{2k+2}}$, so that
\begin{align*}
     \int_{\substack{{0\leq 2\nu_p(t)<\nu_p(s)}\\{2\nmid\nu_p(t)}}}\Big(\frac{-1}{p}\Big)dt
     &=\sum_{0\leq k < \frac{\nu_p(s)-2}{4}}\Big(\frac{-1}{p}\Big)\frac{p-1}{p^{2k+2}}\\
     &=\Big(\frac{-1}{p}\Big)\frac{p-1}{p^2}\sum_{0\leq k < \frac{\nu_p(s)-2}{4}} (p^{-2})^k.
\end{align*}
Now, it is merely a matter of simplifying the geometric sum, taking into account the range of $k$: if $\nu_p(s)=0,1,2$, then the sum is empty and the integral vanishes; if $\nu_p(s)=3,4,5,6$, then the only contribution comes from $k=0$, so that the integral is equal to $\Big(\frac{-1}{p}\Big)\frac{p-1}{p^2}$; for the remaining cases, let $\alpha
=\lfloor{\frac{\nu_p(s)-2}{4}}\rfloor$ and note that
\begin{align*}
    \Big(\frac{-1}{p}\Big)\frac{p-1}{p^2}\sum_{0\leq k < \frac{\nu_p(s)-2}{4}} (p^{-2})^k
    &=\Big(\frac{-1}{p}\Big)\frac{p-1}{p^2}\frac{1}{1-p^{-2}}\cdot
    \left\{\def\arraystretch{1.2}%
\begin{array}{@{}l@{\quad}l@{}}
1-p^{-2\alpha} & \mbox{if $\nu_p(s)\equiv 2 \Mod{4}$,}\\
1-p^{-2\alpha-2} & \mbox{otherwise.}
\end{array}\right.
\end{align*}
We have the following:
\begin{align*}
    \int_{0\leq 2\nu_p(t)<\nu_p(s)}w_p^*(t)dt
    =\left\{\def\arraystretch{1.2}%
\begin{array}{@{}l@{\quad}l@{}}
0 & \mbox{if $\nu_p(s)=0,1,2$,}\\
\Big(\frac{-1}{p}\Big)\frac{p-1}{p^2} & \mbox{if $\nu_p(s)=3,4,5,6,$}\\
\Big(\frac{-1}{p}\Big)\frac{1}{p+1}\cdot
\left\{\def\arraystretch{1.2}%
\begin{array}{@{}l@{\quad}l@{}}
1-p^{-2\alpha} & \mbox{if $\nu_p(s)\equiv 2\Mod{4},$}\\
1-p^{-2\alpha-2} & \mbox{otherwise,}
\end{array}\right.
& \mbox{if $\nu_p(s)\geq 7$.}
\end{array}\right.
\end{align*}

For the integral over $\{t\in\Z_p:0\leq \nu_p(s)<2\nu_p(t)\}$, a quick calculation yields:
\begin{align*}
\int_{0\leq \nu_p(s)< 2\nu_p(t)}w_p^*(t)dt
&=
\left\{\def\arraystretch{1.2}%
\begin{array}{@{}l@{\quad}l@{}}
\int_{0\leq \nu_p(s)<2\nu_p(t)}\Big(\frac{-1}{p}\Big)^\frac{
\nu_p(s)}{2}dt & \mbox{if $\nu_p(s)$ is even,}\\
\int_{0\leq \nu_p(s)<2\nu_p(t)}\Big(\frac{2}{p}\Big)dt & \mbox{if $\nu_p(s)$ is odd,}
\end{array}\right.\\
&    =\left\{\def\arraystretch{1.2}%
\begin{array}{@{}l@{\quad}l@{}}
\Big(\frac{-1}{p}\Big)^\frac{
\nu_p(s)}{2}\frac{1}{p^{\frac{\nu_p(s)}{2}+1}} & \mbox{if $\nu_p(s)$ is even,}\\
\Big(\frac{2}{p}\Big)\frac{1}{p^{\frac{\nu_p(s)+1}{2}}} & \mbox{if $\nu_p(s)$ is odd.}
\end{array}\right.
\end{align*}

Finally, for the integral over $\{t\in\Z_p:0\leq 2\nu_p(t)=\nu_p(s)\}$, we assume $\nu_p(s)$ is even (otherwise, the domain of integration is empty and there is nothing to prove) and we have the following:
\begin{align*}
    \int_{0\leq 2\nu_p(t)=\nu_p(s)}w_p^*(t)dt
    &=\sum_{k=0}^\infty \int_{\substack{{0\leq 2\nu_p(t)=\nu_p(s)}\\{\nu_p(t^2-s)=\nu_p(s)+k}}}
w_p^*(t)dt,
\end{align*}
where, in this case,
\begin{align*}
    w_p^*(t)
    =\left\{\def\arraystretch{1.2}%
\begin{array}{@{}l@{\quad}l@{}}
\Big(\frac{-1}{p}\Big)^{k+1} & \mbox{if $k\equiv 1-\frac{3\nu_p(s)}{2}\Mod{2},$}\\
\Big(\frac{-3}{p}\Big)\Big(\frac{-1}{p}\Big)^{k} & \mbox{if $k\equiv -\frac{3\nu_p(s)}{2}\Mod{2},\not\equiv 0\Mod{3}$}\\
\Big(\frac{-1}{p}\Big)^k & \mbox{if $k\equiv -\frac{3\nu_p(s)}{2}\Mod{6}$}.
\end{array}\right.
\end{align*}
Moreover,
\begin{align*}
    \sum_{k=0}^\infty &\int_{\substack{{0\leq \nu_p(s)=2\nu_p(t)}\\{\nu_p(t^2-s)=\nu_p(s)+k}}}
w_p^*(t)dt\\
&=\sum_{k\equiv 1-\frac{3\nu_p(s)}{2} \Mod{2}}\Big(\frac{-1}{p}\Big)^{k+1}\mu(S_k)
+\sum_{\substack{{k\equiv \frac{-3\nu_p(s)}{2} \Mod{2}}\\{k\not\equiv 0\Mod{3}}}}\Big(\frac{-3}{p}\Big)\Big(\frac{-1}{p}\Big)^k\mu(S_k)
+\sum_{k\equiv \frac{-3\nu_p(s)}{2}\Mod{6}}\Big(\frac{-1}{p}\Big)^k\mu(S_k),
\end{align*}
with $\mu(S_k)$ as in Lemma \ref{lem_meas}. If we let $j\in\{0,2\}$ be such that $\nu_p(s)\equiv j \Mod{4}$, this becomes
\begin{align*}
\Big(\frac{-1}{p}\Big)^{\frac{j}{2}}\Biggr(\sum_{k\equiv 1-\frac{j}{2} \Mod{2}}\mu(S_k)
+\sum_{\substack{{k\equiv \frac{j}{2} \Mod{2}}\\{k\not\equiv 0\Mod{3}}}}\Big(\frac{-3}{p}\Big)\mu(S_k)
+\sum_{k\equiv \frac{3j}{2}\Mod{6}}\mu(S_k)\Biggr).
\end{align*}
In the case where $\Big(\frac{s_p}{p}\Big)=-1$,  
\begin{align*}
    \mu(S_k)=\left\{\def\arraystretch{1.2}%
\begin{array}{@{}l@{\quad}l@{}}
\frac{p-1}{p^{\frac{\nu_p(s)}{2}+1}} & \mbox{if $k=0$,}\\
0 & \mbox{if $k\geq 1$;}
\end{array}\right.
\end{align*}
in particular,
\begin{align*}
    \int_{0\leq \nu_p(s)=2\nu_p(t)}w_p^*(t)dt
=\Big(\frac{-1}{p}\Big)^{\frac{j}{2}}\frac{p-1}{p^{\frac{\nu_p(s)}{2}+1}},
\end{align*}
as the only contribution comes from $\mu(S_0)$. The case where $\Big(\frac{s_p}{p}\Big)=+1$ requires more work. We begin by recalling that 
\begin{align*}
     \mu(S_k)=\left\{\def\arraystretch{1.2}%
\begin{array}{@{}l@{\quad}l@{}}
\frac{p-3}{p^{\frac{\nu_p(s)}{2}+1}} & \mbox{if $k=0$,}\\
\frac{2}{p^{\frac{\nu_p(s)}{2}+k+1}} & \mbox{if $k\geq 1$}.
\end{array}\right.
\end{align*}
By separating $\mu(S_0)$ from $\mu(S_k)$ for $k\geq 1$, we obtain 
 \begin{align*}
&\int_{0\leq \nu_p(s)=2\nu_p(t)}w_p^*(t)dt\\
&=\Big(\frac{-1}{p}\Big)^{\frac{j}{2}}\Biggr(
\mu(S_{1-\frac{j}{2}})
+\mu(S_{\frac{3j}{2}})
+\Big(\frac{-3}{p}\Big)\Biggr(\mu(S_{\frac{j}{2}})-\mu(S_{\frac{3j}{2}})\Biggr)\\
&\qquad \qquad+\sum_{k=1}^\infty\mu(S_{2k+1-\frac{j}{2}})
+\sum_{k= 1}^\infty\mu(S_{6k+\frac{3j}{2}})
+\Big(\frac{-3}{p}\Big)
\Biggr(\sum_{k=1}^\infty\mu(S_{2k+\frac{j}{2}})-\sum_{k= 1}^\infty\mu(S_{6k+\frac{3j}{2}})\Biggr)\Biggr),
\end{align*}
where 
\begin{align*}
    \sum_{k= 0}^\infty \mu(S_{2k+\frac{j}{2}})-\sum_{k=0}^\infty \mu(S_{6k+\frac{3j}{2}})=\sum_{\substack{{k\equiv \frac{j}{2} \Mod{2}}\\{k\not\equiv 0\Mod{3}}}}\mu(S_k).
\end{align*}
For $k\geq 1$, $\mu(S_k)=\frac{2(p-1)}{p^{\frac{\nu_p(s)}{2}+k+1}}$ and it is easy to see that 
\begin{align*}
    &\int_{0\leq \nu_p(s)=2\nu_p(t)}w_p^*(t)dt\\
    &=\Big(\frac{-1}{p}\Big)^{\frac{j}{2}}\Biggr(
\mu(S_{1-\frac{j}{2}})
+\mu(S_{\frac{3j}{2}})
+\Big(\frac{-3}{p}\Big)\Biggr(\mu(S_{\frac{j}{2}})-\mu(S_{\frac{3j}{2}})\Biggr)\\
&\qquad\qquad+\frac{2(p-1)}{p^{\frac{\nu_p(s)}{2}+1}}\Biggr(
\frac{1}{p^{1-\frac{j}{2}}(p^2-1)}
+\frac{1}{p^\frac{3j}{2}(p^6-1)}
+\Big(\frac{-3}{p}\Big)
\Biggr(
\frac{1}{p^\frac{j}{2}(p^2-1)}
-\frac{1}{p^\frac{3j}{2}(p^6-1)}
\Biggr)\Biggr)\Biggr).
\end{align*}
If $p\equiv 1 \Mod{3}$, then $\Big(\frac{-3}{p}\Big)=1$ and we get that
\begin{align*}
      \int_{0\leq \nu_p(s)=2\nu_p(t)}w_p^*(t)dt
      &=\Big(\frac{-1}{p}\Big)^{\frac{j}{2}}\Biggr(
\mu(S_{1-\frac{j}{2}})+
\mu(S_{\frac{j}{2}})
+\frac{2(p-1)}{p^{\frac{\nu_p(s)}{2}+1}}\cdot
\frac{p^\frac{j}{2}+p^{1-\frac{j}{2}}}{p(p^2-1)}
\Biggr).
\end{align*}
Upon further simplification,
\begin{align*}
   \int_{0\leq \nu_p(s)=2\nu_p(t)}w_p^*(t)dt
   &=\Big(\frac{-1}{p}\Big)^{\frac{j}{2}}\frac{p-1}{p^{\frac{\nu_p(s)}{2}+1}}.
\end{align*}
On the other hand, for $p\equiv 2 \Mod{3}$, $\Big(\frac{-3}{p}\Big)=-1$; in particular,
\begin{align*}
    \int_{0\leq \nu_p(s)=2\nu_p(t)}w_p^*(t)dt
    &=\Big(\frac{-1}{p}\Big)^{\frac{j}{2}}\Biggr(
\mu(S_{1-\frac{j}{2}})
+2\mu(S_{\frac{3j}{2}})
-\mu(S_{\frac{j}{2}})
+\frac{2(p-1)}{p^{\frac{\nu_p(s)}{2}+1}}\Biggr(
\frac{p^\frac{j}{2}-p^{1-\frac{j}{2}}}{p(p^2-1)}
+\frac{2}{p^\frac{3j}{2}(p^6-1)}
\Biggr)\Biggr).
\end{align*}
Simplifying once again,
\begin{align*}
    \int_{0\leq \nu_p(s)=2\nu_p(t)}w_p^*(t)dt=
\Big(\frac{-1}{p}\Big)^{\frac{j}{2}}\frac{1}{p^{\frac{\nu_p(s)}{2}+1}}\Biggr(
p-(2j+1)-4(-1)^\frac{j}{2}\frac{p^4+\frac{j}{2}p^3+p^2+\frac{j}{2}}{(p+1)(p^4+p^2+1)}\Biggr),
\end{align*}
which is the desired result.

To complete our proof, it suffices to sum our results, recalling that 
\begin{align*}
\int_{\Z_p}w_p^*(t)dt=\Biggr(\int_{0\leq 2\nu_p(t)<\nu_p(s)}+\int_{0\leq \nu_p(s)<2\nu_p(t)}+\int_{0\leq \nu_p(s)=2\nu_p(t)}\Biggr) w_p^*(t)dt.
\end{align*}
\end{proof}

\section{Computing $\int_{\Z_3}w_3^*(t)dt$}

We begin by recalling that $w_3^*(t)=(-1)^{\nu_3(t^2-s)}w_3(t)$, with $w_3(t)$ as in Appendix A of \cite{BDD}. From here, we consider the usual cases: $0\leq \nu_3(s)<2\nu_3(t),0\leq 2\nu_3(t)<\nu_3(s),0\leq 2\nu_3(t)=\nu_3(s)$.

\subsection{$0\leq \nu_3(s)< 2\nu_3(t)$}

If $0\leq \nu_3(s)<2\nu_3(t)$, then $\nu_3(t^2-s)=\nu_3(s)$ and $w_3^*(t)=(-1)^{\nu_3(s)}w_3(t)$. Since $w_3(t)$ depends only on $\nu_3(t)$ and $t_3 \Mod{3}$ (and possibly on $\nu_3(s)$ and $s_3$), $w_3(t)$ is a $3$-uniformly locally constant multiplicative function with uniformity constant $\eta=1$. Therefore,
\begin{align*}
    \int_{0\leq \nu_3(s)<2\nu_3(t)}w_3^*(t)dt
    &=(-1)^{\nu_3(s)}\sum_{e>\frac{\nu_3(s)}{2}}\Biggr(\frac{1}{3^{e+1}}\sum_{d\in(\Z/3\Z)^*} w_3(d\cdot 3^e)\Biggr)
\end{align*}
and it is not hard to show that

\begin{align*}
    \int_{0\leq \nu_3(s) <2\nu_3(t)} w_3^*(t)dt =
    \left\{\def\arraystretch{1.2}%
\begin{array}{@{}l@{\quad}l@{}}
\frac{1}{3^{\frac{\nu_3(s)}{2}+2}} & \mbox{if $\nu_3(s)\equiv 0 \Mod{2}$,}\\
\frac{1-2\chi_3(s_3)}{3^{\frac{\nu_3(s)+3}{2}}} & \mbox{if $\nu_3(s)\equiv 1 \Mod{4},$}\\
\frac{-1}{3^{\frac{\nu_3(s)+1}{2}}} & \mbox{if $\nu_3(s)\equiv 3\Mod{4}$},
\end{array}\right.
\end{align*}
where $\chi_3$ is the non-principal character modulo $3$.

\subsection{$0\leq 2\nu_3(t)<\nu_3(s)$}

If $0\leq 2\nu_3(t)<\nu_3(s)$, then $\nu_3(t^2-s)=2\nu_3(t)$ and $w_3^*(t)=w_3(t)$. Once again, $w_3(t)$ is a $3$-uniformly locally constant multiplicative function with uniformity constant $\eta=1$.
We begin by partitioning the integral $\int_{0\leq 2\nu_3(t) <\nu_3(s)}w_3^*(t)dt$ according to the cases in Appendix A of \cite{BDD}:
\begin{align*}
    &\int_{0\leq 2\nu_3(t) <\nu_3(s)}w_3^*(t)dt\\
    &=
    \int_{\substack{{\nu_3(s)-2\nu_3(t)=1}\\{2|\nu_3(t)}}}w_3(t)dt
    +\int_{\substack{{\nu_3(s)-2\nu_3(t)=2}\\{2|\nu_3(t)}}}w_3(t)dt
    +\int_{\substack{{\nu_3(s)-2\nu_3(t)\geq 3}\\{2|\nu_3(t)}}}w_3(t)dt\\
    &\qquad \qquad+\int_{\substack{{\nu_3(s)-2\nu_3(t)=1}\\{2\nmid \nu_3(t)}}}w_3(t)dt
    +\int_{\substack{{\nu_3(s)-2\nu_3(t)=2}\\{2\nmid \nu_3(t)}}}w_3(t)dt
    +\int_{\substack{{\nu_3(s)-2\nu_3(t)=3}\\{2\nmid\nu_3(t)}}}w_3(t)dt
    +\int_{\substack{{\nu_3(s)-2\nu_3(t)\geq 4}\\{2\nmid\nu_3(t)}}}w_3(t)dt,
\end{align*}

From Appendix A in \cite{BDD}, 
\begin{align*}
    \int_{\substack{{\nu_3(s)-2\nu_3(t)=2}\\{2|\nu_3(t)}}}w_3(t)dt,
    \int_{\substack{{\nu_3(s)-2\nu_3(t)=2}\\{2\nmid \nu_3(t)}}}w_3(t)dt,
    \int_{\substack{{\nu_3(s)-2\nu_3(t)\geq 4}\\{2\nmid\nu_3(t)}}}w_3(t)dt=0,
\end{align*}
whereas
\begin{align*}
    &\int_{\substack{{\nu_3(s)-2\nu_3(t)=1}\\{2|\nu_3(t)}}}w_3(t)dt
=\left\{\def\arraystretch{1.2}%
\begin{array}{@{}l@{\quad}l@{}}
\frac{2}{3^{\frac{\nu_3(s)+1}{2}}} & \mbox{if $\nu_3(s)\equiv 1 \Mod{4}$ and $\nu_3(s)\geq 1$,}\\
0 & \mbox{otherwise},
\end{array}\right.\\
    &\int_{\substack{{\nu_3(s)-2\nu_3(t)\geq 3}\\{2|\nu_3(t)}}}w_3(t)dt
=\sum_{\substack{{3\leq k \leq \nu_3(s)}\\{k\equiv \nu_3(s) \Mod{4}}}} \frac{-2}{3^{\frac{\nu_3(s)-k}{2}+1}}    
=\left\{\def\arraystretch{1.2}%
\begin{array}{@{}l@{\quad}l@{}}
\frac{1}{4}\Biggr(\frac{3^{1-2\lfloor \frac{j}{3} \rfloor}}{3^{\frac{\nu_3(s)-j}{2}}}-3\Biggr) & \mbox{if $\nu_3(s)\geq 3$,}\\
0 & \mbox{otherwise},
\end{array}\right.\\
    &\int_{\substack{{\nu_3(s)-2\nu_3(t)=1}\\{2\nmid \nu_3(t)}}}w_3(t)dt
=\left\{\def\arraystretch{1.2}%
\begin{array}{@{}l@{\quad}l@{}}
\frac{2\chi_3(s_3)}{3^{\frac{\nu_3(s)+1}{2}}} & \mbox{if $\nu_3(s) \equiv 3 \Mod{4}$ and $\nu_3(s)\geq 3$,}\\
0 & \mbox{otherwise},
\end{array}\right.\\
    &\int_{\substack{{\nu_3(s)-2\nu_3(t)=3}\\{2\nmid\nu_3(t)}}}w_3(t)dt
=\left\{\def\arraystretch{1.2}%
\begin{array}{@{}l@{\quad}l@{}}
\frac{2}{3^{\frac{\nu_3(s)-1}{2}}} & \mbox{if $\nu_3(s) \equiv 1 \Mod{4}$ and $\nu_3(s)\geq 5$,}\\
0 & \mbox{otherwise},
\end{array}\right.
\end{align*}
where $j\in \{0,1,2,3\}$ is such that $\nu_3(s)\equiv j \Mod{4}$ and where $\chi_3$ is the non-principal character modulo $3$.

Summing the individual contributions, 
\begin{align*}
    \int_{0\leq 2\nu_3(t)< \nu_3(s)}w_3^*(t)dt
    =
\left\{\def\arraystretch{1.2}%
\begin{array}{@{}l@{\quad}l@{}}
0 & \mbox{if $\nu_3(s)=0$,}\\
\frac{2}{3} & \mbox{if $\nu_3(s)=1$,}\\
0 & \mbox{if $\nu_3(s)=2$,}\\
\frac{2(\chi_3(s_3)-3)}{9}& \mbox{if $\nu_3(s)=3,$}\\
\frac{-2}{3} & \mbox{if $\nu_3(s)=4$,}\\
\frac{1}{4}\Biggr(\frac{3^{1-2\lfloor \frac{j}{3} \rfloor}}{3^{\frac{\nu_3(s)-j}{2}}}-3\Biggr)
+
\left\{\def\arraystretch{1.2}%
\begin{array}{@{}l@{\quad}l@{}}
 0 & \mbox{if $\nu_3(s)\equiv 0 \Mod{2}$,}\\
 \frac{8}{3^{\frac{\nu_3(s)+1}{2}}}& \mbox{if $\nu_3(s)\equiv 1 \Mod{4},$}\\
 \frac{2\chi_3(s_3)}{3^{\frac{\nu_3(s)+1}{2}}}& \mbox{if $\nu_3(s)\equiv 3 \Mod{4},$}
\end{array}\right.
& \mbox{if $\nu_3(s)\geq 5$.}
\end{array}\right.
\end{align*}

\subsection{$0\leq 2\nu_3(t)=\nu_3(s)$}

For $0\leq 2\nu_3(t)=\nu_3(s)$, we write $\nu_3(t^2-s)=\nu_3(s)+k$ with $k\geq 0$, so that
\begin{align*}
    \int_{2\nu_3(t)=\nu_3(s)}w_3^*(t)dt
    &=\sum_{k=0}^\infty (-1)^k \int_{\substack{{0\leq 2\nu_3(t)=\nu_3(s)}\\{\nu_3(t^2-s)=\nu_3(s)+k}}}w_3(t)dt.
\end{align*}
By splitting the contributions from $k=0, k\not \equiv0  \Mod{3}$, and $k\equiv 0 \Mod{3} (k\neq 0)$, we write
\begin{align*}
    &\int_{2\nu_3(t)=\nu_3(s)}w_3^*(t)dt\\
    &=
    \int_{\substack{{0\leq 2\nu_3(t)=\nu_3(s)}\\{\nu_3(t^2-s)=\nu_3(s)}}}w_3(t)dt
    +\sum_{\substack{{k\equiv 0 \Mod{3}}\\{k\neq 0}}} (-1)^k\int_{\substack{{0\leq 2\nu_3(t)=\nu_3(s)}\\{\nu_3(t^2-s)=\nu_3(s)+k}}}w_3(t)dt
    +\sum_{k\not \equiv 0 \Mod{3}}(-1)^k \int_{\substack{{0\leq 2\nu_3(t)=\nu_3(s)}\\{\nu_3(t^2-s)=\nu_3(s)+k}}}w_3(t)dt.
\end{align*}
Notice that if $2\nu_3(t)=\nu_3(s)$, then $\nu_3(t^2-s)=\nu_3(s)+k$ iff $t_3^2-s_3\in 3^k\Z_3^*$; in other words, $ \nu_3(t^2-s)=\nu_3(s)+k$ iff
\begin{align*}
   \left\{\def\arraystretch{1.2}%
\begin{array}{@{}l@{\quad}l@{}}
t_3^2\not \equiv s_3 \Mod{3} & \mbox{if $k=0$,}\\
t_3^2\equiv s_3 \Mod{3^k},\not \equiv s_3 \Mod{3^{k+1}} & \mbox{if $k\geq 1$.}
\end{array}\right.
\end{align*}

Since $w_3^*(t)=(-1)^{\nu_3(t^2-s)}w_3(t)$ and since $w_3(t)$ depends only on $t_3(t_3^2-s_3)_3 \Mod{9}$ (and possibly on $s_3$ and $\nu_3(s)$), we have that
\begin{align*}
    \int_{\substack{{0\leq 2\nu_3(t)=\nu_3(s)}\\{\nu_3(t^2-s)=\nu_3(s)+k}}}w_3(t)dt
    =\frac{1}{3^{\frac{\nu_3(s)}{2}+k+2}}\sum_{\substack{{d\in(\Z/3^{k+2}\Z)^*}\\{d^2\equiv s_3\Mod{3^k}}\\{d^2\not\equiv s_3\Mod{3^{k+1}}}}}w_3(d\cdot 3^\frac{\nu_3(s)}{2}).
\end{align*}
We consider two cases: $s_3\equiv 1 \Mod{3}$ and $s_3\equiv 2 \Mod{3}$.

In the case where $s_3\equiv 2\Mod{3}$, $s_3$ is not a square modulo $3$; in particular,
\begin{align*}
    \sum_{\substack{{d\in(\Z/3^{k+2}\Z)^*}\\{d^2\equiv s_3\Mod{3^k}}\\{d^2\not\equiv s_3\Mod{3^{k+1}}}}}w_3(d\cdot 3^\frac{\nu_3(s)}{2})=0
\end{align*}
for all $k\geq 1 $ (as the sums are empty). Therefore, if $s_3\equiv 2 \Mod{3}$,
\begin{align*}
     \int_{2\nu_3(t)=\nu_3(s)}w_3^*(t)dt
     &=\frac{1}{3^{\frac{\nu_3(s)}{2}+2}}\sum_{\substack{{d\in(\Z/3^{2}\Z)^*}\\{d^2\not\equiv 2\Mod{3}}}}w_3(d\cdot 3^\frac{\nu_3(s)}{2})\\
     &=\frac{1}{3^{\frac{\nu_3(s)}{2}+2}}\sum_{\substack{{d\in(\Z/3^{2}\Z)^*}}}w_3(d\cdot 3^\frac{\nu_3(s)}{2}).
\end{align*}
In this case, $w_3(d\cdot 3^\frac{\nu_3(s)}{2})=1$ iff $s_3d\not \equiv 2,4 \Mod{9}$. Since $s_3$ is invertible modulo $9$, as $d$ varies over $(\Z/9\Z)^*$, so does $s_3d$; i.e.,
\begin{align*}
    \sum_{\substack{{d\in(\Z/3^{2}\Z)^*}}}w_3(d\cdot 3^\frac{\nu_3(s)}{2})=2
\end{align*}
with
\begin{align*}
    \int_{2\nu_3(t)=\nu_3(s)}w_3^*(t)dt=\frac{2}{3^{\frac{\nu_3(s)}{2}+2}} \mbox{ if $s_3\equiv 2 \Mod{3}$}.
\end{align*}

In the case where $s_3\equiv 1 \Mod{3}$, let $\pm\sqrt s_3$ denote the square roots of $s_3$ in $\Z_3$. Since $s_3$ is a square modulo $3$, there exist exactly two $d$ in $(\Z/3^k\Z)^*\cong (\Z_3/3^k\Z_3)^*$ such that $d^2\equiv s_3 \Mod{3^k}$ (namely, $\pm \sqrt s_3 + 3^k\Z_3$). Each such solution lifts in exactly three ways to solutions of $x^2\equiv s_3 \Mod{3^k}$ in $(\Z/3^{k+1}\Z)^*$; namely, $\pm(\sqrt s_3+\alpha \cdot 3^k)+3^{k+1}\Z_3$ with $\alpha \in \{0,1,2\}$. The condition that $x^2\not\equiv s_3 \Mod{3^{k+1}}$ tells us to throw away two of our solutions (those corresponding to $\alpha=0$). From here, we lift our solutions to $(\Z/3^{k+2}\Z)^*$ by writing $\pm(\sqrt s_3+\alpha \cdot 3^k+\beta \cdot 3^{k+1})+3^{k+2}\Z_3$ with $\beta \in \{0,1,2\}$. By working with the isomorphism $(\Z/3^{k+2}\Z)^*\cong (\Z_3/3^{k+2}\Z_3)^*$ and choosing an appropriate representative for  $d$, we have that there are exactly 12 solutions to $d\in (\Z/3^{k+2}\Z)^*$ such that $d^2\equiv s_3 \Mod{3^k},\not\equiv s_3 \Mod{3^{k+1}}$; namely,
\begin{align*}
    d=\pm(\sqrt s_3 + \alpha \cdot 3^k + \beta \cdot 3^{k+1}) + 3^{k+2}\Z_3,
\end{align*}
with $\alpha \in \{1,2\},\beta \in \{0,1,2\}$. Now, the value of $w_3(d\cdot 3^\frac{\nu_3(s)}{2})$ depends only on the value of $d(d^2-s_3)_3$ modulo $9$, with $d$ as above (in the case where $k\equiv 0 \Mod{3}$, the value of $w_3(d\cdot 3^\frac{\nu_3(s)}{2})$ depends only on $d(d^2-s_3)_3$ modulo $3$). But, if $d=\pm(\sqrt s_3 + \alpha \cdot 3^k + \beta \cdot 3^{k+1}) + 3^{k+2}\Z_3$, then, for $k\geq 1$,
\begin{align*}
d(d^2-s_3)_3
&\equiv \left\{\def\arraystretch{1.2}%
\begin{array}{@{}l@{\quad}l@{}}
\pm 2 s_3 (\alpha + 3\beta) \Mod{9} & \mbox{if $k\equiv 0 \Mod{3}$,}\\
\pm 2 s_3 \alpha \Mod{3} & \mbox{if $k\not \equiv 0 \Mod{3}$.}
\end{array}\right.
\end{align*}
From here, it is easy to see that
\begin{align*}
    \frac{1}{3^{\frac{\nu_3(s)}{2}+k+2}}\sum_{\substack{{d\in(\Z/3^{k+2}\Z)^*}\\{d^2\equiv s_3 \Mod{3^k}}\\{d^2\not\equiv s_3\Mod{3^{k+1}}}}}w_3(d\cdot 3^\frac{\nu_3(s)}{2})
    =\left\{\def\arraystretch{1.2}%
\begin{array}{@{}l@{\quad}l@{}}
0 & \mbox{if $k\not \equiv 0 \Mod{3}$,}\\
\frac{4}{3^{\frac{\nu_3(s)}{2}+k+2}} & \mbox{otherwise,}
\end{array}\right.
\end{align*}
whenever $k\geq 1$. When $k=0$,
\begin{align*}
     \int_{\substack{{0\leq 2\nu_3(t)=\nu_3(s)}\\{\nu_3(t^2-s)=\nu_3(s)}}}w_3(t)dt=0,
\end{align*}
as the sum
\begin{align*}
    \sum_{\substack{{d\in (\Z/9\Z)^*}\\{d^2\not \equiv s_3 \Mod{3}}}}w_3(d\cdot 3^\frac{\nu_3(s)}{2})
\end{align*}
is empty (simply note that $d^2\equiv 1 \Mod{3}$ for all $d\in (\Z/9\Z)^*$). Putting all of this together, 
\begin{align*}
    \int_{2\nu_3(t)=\nu_3(s)}w_3^*(t)dt
    &=\sum_{\substack{{k\equiv 0 \Mod{3}}\\{k\neq 0}}} (-1)^k\frac{4}{3^{\frac{\nu_3(s)}{2}+k+2}} 
    =\frac{-1}{7}\cdot\frac{1}{3^{\frac{\nu_3(s)}{2}+2}};
\end{align*}
that is,
\begin{align*}
      \int_{0\leq 2\nu_3(t)=\nu_3(s)}w_3^*(t)dt=
      \left\{\def\arraystretch{1.2}%
\begin{array}{@{}l@{\quad}l@{}}
\left\{\def\arraystretch{1.2}%
\begin{array}{@{}l@{\quad}l@{}}
\frac{2}{3^{\frac{\nu_3(s)}{2}+2}} & \mbox{if $s_3\equiv 2 \Mod{3}$,}\\
\frac{-1}{7}\cdot\frac{1}{3^{\frac{\nu_3(s)}{2}+2}} & \mbox{if $s_3\equiv 1 \Mod{3}$},
\end{array}\right.
& \mbox{if $\nu_3(s)\equiv 0 \Mod{2}$,}\\
0 & \mbox{if $\nu_3(s)\equiv 1 \Mod{2}$}.
\end{array}\right.
\end{align*}

Hence,
\begin{prop}
\label{p=3}
\begin{align*}
     \int_{\Z_3}w_3^*(t)dt
&=\left\{\def\arraystretch{1.2}%
\begin{array}{@{}l@{\quad}l@{}}
\frac{1}{3^{\frac{\nu_3(s)}{2}+2}} & \mbox{if $\nu_3(s)\equiv 0 \Mod{2}$,}\\
\frac{1-2\chi_3(s_3)}{3^{\frac{\nu_3(s)+3}{2}}} & \mbox{if $\nu_3(s)\equiv 1 \Mod{4}$,}\\
\frac{-1}{3^{\frac{\nu_3(s)+1}{2}}} & \mbox{if $\nu_3(s)\equiv 3\Mod{4}$},
\end{array}\right.\\
&+
\left\{\def\arraystretch{1.2}%
\begin{array}{@{}l@{\quad}l@{}}
0 & \mbox{if $\nu_3(s)=0$,}\\
\frac{2}{3} & \mbox{if $\nu_3(s)=1,$}\\
0 & \mbox{if $\nu_3(s)=2$,}\\
\frac{2(\chi_3(s_3)-3)}{9}& \mbox{if $\nu_3(s)=3,$}\\
\frac{-2}{3} & \mbox{if $\nu_3(s)=4$,}\\
\frac{1}{4}\Biggr(\frac{3^{1-2\lfloor \frac{j}{3} \rfloor}}{3^{\frac{\nu_3(s)-j}{2}}}-3\Biggr)
+
\left\{\def\arraystretch{1.2}%
\begin{array}{@{}l@{\quad}l@{}}
 0 & \mbox{if $\nu_3(s)\equiv 0 \Mod{2}$,}\\
 \frac{8}{3^{\frac{\nu_3(s)+1}{2}}}& \mbox{if $\nu_3(s)\equiv 1 \Mod{4},$}\\
 \frac{2\chi_3(s_3)}{3^{\frac{\nu_3(s)+1}{2}}}& \mbox{if $\nu_3(s)\equiv 3 \Mod{4},$}
\end{array}\right.
& \mbox{if $\nu_3(s)\geq 5$,}
\end{array}\right.\\
&+\left\{\def\arraystretch{1.2}%
\begin{array}{@{}l@{\quad}l@{}}
\left\{\def\arraystretch{1.2}%
\begin{array}{@{}l@{\quad}l@{}}
\frac{2}{3^{\frac{\nu_3(s)}{2}+2}} & \mbox{if $s_3\equiv 2 \Mod{3}$,}\\
\frac{-1}{7}\cdot\frac{1}{3^{\frac{\nu_3(s)}{2}+2}} & \mbox{if $s_3\equiv 1 \Mod{3}$},
\end{array}\right.
& \mbox{if $\nu_3(s)\equiv 0 \Mod{2}$,}\\
0 & \mbox{if $\nu_3(s)\equiv 1 \Mod{2}$},
\end{array}\right.
\end{align*}
where $j\in \{0,1,2,3\}$ is such that $\nu_3(s)\equiv j \Mod{4}$ and where $\chi_3$ is the non-principal character modulo $3$.
\end{prop}

\section{Computing $\int_{\Z_2}w_2^*(t)dt$}

We begin by recalling that $w_2^*(t)\in\{\pm 1\}$ with $w_2^*(t)\equiv (t^2-s)_2w_2(t) \Mod{4}$. We consider the usual cases: $0\leq \nu_2(s)<2\nu_2(t),0\leq 2\nu_2(t)<\nu_2(s)$, and $0\leq 2\nu_2(t)=\nu_2(s)$.

\subsection{$0\leq \nu_2(s)<2\nu_2(t)$}

If $0\leq \nu_2(s)<2\nu_2(t)$, then $\nu_2(t^2-s)=\nu_2(s)$ and $2\nu_2(t)=\nu_2(s)+k$, for some $k\geq 1$; in particular,
\begin{align*}
(t^2-s)_2
&=(t^2-s)2^{-\nu_2(s)}\\
&=t_2^2\cdot 2^k-s_2\\
&\equiv 
\left\{\def\arraystretch{1.2}%
\begin{array}{@{}l@{\quad}l@{}}
s_2 \Mod{4} & \mbox{if $k=1,$}\\
-s_2 \Mod{4} & \mbox{if $k\geq 2$}.
\end{array}\right.
\end{align*}
Therefore,
\begin{align*}
\int_{2\nu_2(t)-\nu_2(s)=k} w_2^*(t)dt
&=\int_{2\nu_2(t)-\nu_2(s)=k} w_2(t)dt \cdot
\left\{\def\arraystretch{1.2}%
\begin{array}{@{}l@{\quad}l@{}}
\chi_4(s_2) & \mbox{if $k=1$,}\\
-\chi_4(s_2) & \mbox{if $k\geq 2$},
\end{array}\right.
\end{align*}
where $\chi_4$ is the non-principal character modulo $4$.

Since $w_2(t)$ depends only on $\nu_2(t)$ and $t_2 \Mod{4}$ (and possibly on $\nu_2(s)$ and $s_2$), we have that $w_2(t)$ is a $2$-uniformly locally constant multiplicative function with uniformity constant $\eta=2$; i.e.,
\begin{align*}
\int_{2\nu_2(t)-\nu_2(s)=k} w_2(t)dt
&=\frac{1}{2^{\frac{\nu_2(s)+k}{2}+2}}\sum_{d\in (\Z/4\Z)^*} w_2(d\cdot 2^{\frac{\nu_2(s)+k}{2}}).
\end{align*}
Putting all of this together,
\begin{align*}
\int_{0\leq \nu_2(s)<2\nu_2(t)}w_2^*(t)dt
&=\chi_4(s_2)\cdot
\left\{\def\arraystretch{1.2}%
\begin{array}{@{}l@{\quad}l@{}}
-\sum_{k=1}^\infty \int_{\nu_2(t)=\frac{\nu_2(s)}{2}+k} w_2(t)dt & \mbox{if $\nu_2(s)\equiv 0 \Mod{2},$}\\
\int_{\nu_2(t)=\frac{\nu_2(s)+1}{2}}w_2(t)dt-\sum_{k=1}^\infty \int_{\nu_2(t)=\frac{\nu_2(s)+1}{2}+k} w_2(t)dt & \mbox{if $\nu_2(s)\equiv 1 \Mod{2}$},
\end{array}\right.
\end{align*}
where $\chi_4$ is the non-principal character modulo $4$ and with $\int_{\nu_2(t)=e}w_2(t)dt$ as above.

From here, a tedious, but straightforward, computation yields:
\begin{align*}
    \int_{0\leq \nu_2(s)<2\nu_2(t)}w_2^*(t)dt
    =\left\{\def\arraystretch{1.2}%
\begin{array}{@{}l@{\quad}l@{}}
0 & \mbox{if $\nu_2(s)\equiv 0 \Mod{2},$}\\
\frac{(-1)^\frac{\nu_2(s)-1}{2}}{2^\frac{\nu_2(s)+3}{2}}
\cdot \left\{\def\arraystretch{1.2}%
\begin{array}{@{}l@{\quad}l@{}}
1 & \mbox{if $s_2\equiv 1,7 \Mod{8}$,}\\
-1 & \mbox{if $s_2 \equiv 3,5 \Mod{8}$},
\end{array}\right.
& \mbox{if $\nu_2(s)\equiv 1  \Mod{2}$}.
\end{array}\right.
\end{align*}

\subsection{$0\leq 2\nu_2(t)<\nu_2(s)$}

If $0\leq 2\nu_2(t)<\nu_2(s)$, then $\nu_2(t^2-s)=2\nu_2(t)$ and $\nu_2(s)=2\nu_2(t)+k$, for some $k\geq 1$; in particular,
\begin{align*}
(t^2-s)_2
&=(t^2-s)2^{-2\nu_2(t)}\\
&=t_2^2-s_2\cdot 2^k\\
&\equiv 
\left\{\def\arraystretch{1.2}%
\begin{array}{@{}l@{\quad}l@{}}
-1 \Mod{4} & \mbox{if $k=1$,}\\
1 \Mod{4} & \mbox{if $k\geq 2$}.
\end{array}\right.
\end{align*}
Therefore,
\begin{align*}
    \int_{\nu_2(s)-2\nu_2(t)=k}w_2^*(t)dt
    &=\int_{\nu_2(s)-2\nu_2(t)=k}w_2(t)dt
   \cdot \left\{\def\arraystretch{1.2}%
\begin{array}{@{}l@{\quad}l@{}}
-1 & \mbox{if $k=1$,}\\
1 & \mbox{if $k\geq 2$}.
\end{array}\right.
\end{align*}

Since $w_2(t)$ depends only on $\nu_2(t)$ and $t_2 \Mod{8}$ (and possibly on $\nu_2(s)$ and $s_2$), $w_2(t)$ is a $2$-uniformly locally constant multiplicative function with uniformity constant $\eta=3$; that is,
\begin{align*}
    \int_{\nu_2(s)-2\nu_2(t)=k}w_2(t)dt
    &=\frac{1}{2^{\frac{\nu_2(s)-k}{2}+3}}\sum_{d\in (\Z/8\Z)^*} w_2(d\cdot 2^\frac{\nu_2(s)-k}{2}),
\end{align*}
with
\begin{align*}
    &\int_{0\leq 2\nu_2(t)<\nu_2(s)}w_2^*(t)dt\\
    &=-\int_{\substack{{\nu_2(s)-2\nu_2(t)=1}\\{2|\nu_2(t)}}}w_2(t)dt
    +\int_{\substack{{\nu_2(s)-2\nu_2(t)=2}\\{2|\nu_2(t)}}}w_2(t)dt
    +\int_{\substack{{\nu_2(s)-2\nu_2(t)=3}\\{2|\nu_2(t)}}}w_2(t)dt\\
    &+\int_{\substack{{\nu_2(s)-2\nu_2(t)=4}\\{2|\nu_2(t)}}}w_2(t)dt
    +\int_{\substack{{\nu_2(s)-2\nu_2(t)=5}\\{2|\nu_2(t)}}}w_2(t)dt
    +\int_{\substack{{\nu_2(s)-2\nu_2(t)=6}\\{2|\nu_2(t)}}}w_2(t)dt
    +\int_{\substack{{\nu_2(s)-2\nu_2(t)\geq7}\\{2|\nu_2(t)}}}w_2(t)dt\\
   &-\int_{\substack{{\nu_2(s)-2\nu_2(t)=1}\\{2\nmid\nu_2(t)}}}w_2(t)dt
   +\int_{\substack{{\nu_2(s)-2\nu_2(t)=2}\\{2\nmid\nu_2(t)}}}w_2(t)dt
   +\int_{\substack{{\nu_2(s)-2\nu_2(t)=3}\\{2\nmid\nu_2(t)}}}w_2(t)dt
   +\int_{\substack{{\nu_2(s)-2\nu_2(t)\geq4}\\{2\nmid\nu_2(t)}}}w_2(t)dt,
\end{align*}
where we partitioned the integral according to the cases in Appendix A of \cite{BDD}. From Appendix A in \cite{BDD}, it is easy to see that 
\begin{align*}
   &\int_{\substack{{\nu_2(s)-2\nu_2(t)=1}\\{2|\nu_2(t)}}}w_2(t)dt,
   \int_{\substack{{\nu_2(s)-2\nu_2(t)=3}\\{2|\nu_2(t)}}}w_2(t)dt,
   \int_{\substack{{\nu_2(s)-2\nu_2(t)=6}\\{2|\nu_2(t)}}}w_2(t)dt,\\
   &\int_{\substack{{\nu_2(s)-2\nu_2(t)=1}\\{2\nmid\nu_2(t)}}}w_2(t)dt,
   \int_{\substack{{\nu_2(s)-2\nu_2(t)=2}\\{2\nmid\nu_2(t)}}}w_2(t)dt,
   \int_{\substack{{\nu_2(s)-2\nu_2(t)\geq4}\\{2\nmid\nu_2(t)}}}w_2(t)dt=0,
\end{align*}
whereas
\begin{align*}
 &\int_{\substack{{\nu_2(s)-2\nu_2(t)=2}\\{2|\nu_2(t)}}}w_2(t)dt
=\left\{\def\arraystretch{1.2}%
\begin{array}{@{}l@{\quad}l@{}}
\frac{1}{2^{\frac{\nu_2(s)}{2}+1}}
\cdot \left\{\def\arraystretch{1.2}%
\begin{array}{@{}l@{\quad}l@{}}
1 & \mbox{if $s_2\equiv 1 \Mod{4},$}\\
-2 & \mbox{if $s_2 \equiv 3 \Mod{4},$}
\end{array}\right.
& \mbox{if $\nu_2(s)\equiv 2 \Mod{4}$ and $\nu_2(s)\geq 2,$}\\
0 & \mbox{otherwise},
\end{array}\right.\\
&\int_{\substack{{\nu_2(s)-2\nu_2(t)=4}\\{2|\nu_2(t)}}}w_2(t)dt
=\left\{\def\arraystretch{1.2}%
\begin{array}{@{}l@{\quad}l@{}}
\frac{1}{2^{\frac{\nu_2(s)}{2}}} & \mbox{if $\nu_2(s)\equiv 0 \Mod{4}$ and $\nu_2(s)\geq 4,$}\\
0 & \mbox{otherwise},
\end{array}\right.\\
&\int_{\substack{{\nu_2(s)-2\nu_2(t)=5}\\{2|\nu_2(t)}}}w_2(t)dt
=\left\{\def\arraystretch{1.2}%
\begin{array}{@{}l@{\quad}l@{}}
\frac{1}{2^{\frac{\nu_2(s)-1}{2}}} & \mbox{if $\nu_2(s)\equiv 1 \Mod{4}$ and $\nu_2(s)\geq 5,$}\\
0 & \mbox{otherwise},
\end{array}\right.\\
&\int_{\substack{{\nu_2(s)-2\nu_2(t)\geq7}\\{2|\nu_2(t)}}}w_2(t)dt
=\sum_{\substack{{ 7\leq k \leq \nu_2(s)}\\{k\equiv \nu_2(s) \Mod{4}}}}\frac{-2}{2^{\frac{\nu_2(s)-k}{2}+3}}
=\left\{\def\arraystretch{1.2}%
\begin{array}{@{}l@{\quad}l@{}}
\frac{1}{3}\Biggr(\frac{2^{2\lceil \frac{7-j}{4}\rceil}}{2^{\frac{\nu_2(s)-j}{2}+2}}-1\Biggr) & \mbox{if $\nu_2(s)\geq 7,$}\\
0 & \mbox{otherwise,}
\end{array}\right.\\
&\int_{\substack{{\nu_2(s)-2\nu_2(t)=3}\\{2\nmid\nu_2(t)}}}w_2(t)dt
=\left\{\def\arraystretch{1.2}%
\begin{array}{@{}l@{\quad}l@{}}
\frac{-\chi_4(s_2)}{2^{\frac{\nu_2(s)-1}{2}}}
& \mbox{if $\nu_2(s)\equiv 1 \Mod{4}$ and $\nu_2(s)\geq 5$},\\
0 & \mbox{otherwise}.
\end{array}\right.
\end{align*}

Summing the individual contributions,
\begin{align*}
    &\int_{0\leq 2\nu_2(t)< \nu_2(s)}w_2^*(t)dt\\
    &=
\left\{\def\arraystretch{1.2}%
\begin{array}{@{}l@{\quad}l@{}}
0 & \mbox{if $\nu_2(s)=0,$}\\
0 & \mbox{if $\nu_2(s)=1,$}\\
\frac{1}{4}
\cdot \left\{\def\arraystretch{1.2}%
\begin{array}{@{}l@{\quad}l@{}}
1 & \mbox{if $s_2\equiv 1 \Mod{4},$}\\
-2 & \mbox{if $s_2 \equiv 3 \Mod{4},$}
\end{array}\right.
&\mbox{if $\nu_2(s)=2,$}\\
0 & \mbox{if $\nu_2(s)=3,$}\\
\frac{1}{4} & \mbox{if $\nu_2(s)=4$,}\\
\frac{1-\chi_4(s_2)}{4} & \mbox{if $\nu_2(s)=5,$}\\
\frac{1}{16}
\cdot \left\{\def\arraystretch{1.2}%
\begin{array}{@{}l@{\quad}l@{}}
1 & \mbox{if $s_2\equiv 1 \Mod{4},$}\\
-2 & \mbox{if $s_2 \equiv 3 \Mod{4},$}
\end{array}\right. & \mbox{if $\nu_2(s)=6,$}\\
\frac{1}{3}\Biggr(\frac{2^{2\lceil \frac{7-j}{4}\rceil}}{2^{\frac{\nu_2(s)-j}{2}+2}}-1\Biggr)
+\frac{1}{2^\frac{\nu_2(s)-j}{2}}
\left\{\def\arraystretch{1.2}%
\begin{array}{@{}l@{\quad}l@{}}
1 & \mbox{if $\nu_2(s)\equiv 0 \Mod{4},$}\\
1-\chi_4(s_2) & \mbox{if $\nu_2(s)\equiv 1 \Mod{4},$}\\
 \left\{\def\arraystretch{1.2}%
\begin{array}{@{}l@{\quad}l@{}}
\frac{1}{4} & \mbox{if $s_2\equiv 1 \Mod{4},$}\\
\frac{-1}{2}& \mbox{if $s_2 \equiv 3 \Mod{4},$}
\end{array}\right.
& \mbox{if $\nu_2(s)\equiv 2 \Mod{4}$,}\\
0 & \mbox{if $\nu_2(s)\equiv 3 \Mod{4}$,}
\end{array}\right.
& \mbox{if $\nu_2(s)\geq 7$,}
\end{array}\right.
\end{align*}
where $j\in \{0,1,2,3\}$ is such that $\nu_2(s)\equiv j \Mod{4}$ and where $\chi_4$ is the non-principal character modulo $4$.

\subsection{$0\leq 2\nu_2(t)=\nu_2(s)$}

To deal with the case where $0\leq 2\nu_2(t)=\nu_2(s)$, we first write
\begin{align*}
    \int_{0\leq 2\nu_2(t)=\nu_2(s)} w_2^*(t)dt =\sum_{k=0}^\infty \int_{\substack{{2\nu_2(t)=\nu_2(s)}\\{\nu_2(t^2-s)=\nu_2(s)+k}}}w_2^*(t)dt,
\end{align*}
with $w_2^*(t)\in \{\pm 1\}$ such that
\begin{align*}
    w_2^*(t)\equiv (t^2-s)_2w_2(t) \Mod{4},
\end{align*}
and where $w_2(t)$ is given by Appendix A in \cite{BDD}.

Since $w_2^*(t)$ depends only on $\nu_2(t)$, $t_2 \Mod{8}$, and $(t^2-s)_2 \Mod{8}$, we have that
\begin{align*}
    \int_{\substack{{2\nu_2(t)=\nu_2(s)}\\{\nu_2(t^2-s)=\nu_2(s)+k}}}w_2^*(t)dt=\frac{1}{2^{\frac{\nu_2(s)}{2}+k+3}}\sum_{\substack{{d\in (\Z/2^{k+3}\Z)^*}\\{d^2\equiv s_2 \Mod{2^k}}\\{d^2\not \equiv s_2 \Mod{2^{k+1}}}}}(d^2-s_2)_2^\prime w_2(d \cdot 2^\frac{\nu_2(s)}{2}),
\end{align*}
where the $\prime$ indicates that we take $(d^2-s_2)_2^\prime$ in $\{\pm 1\}$ such that $(d^2-s_2)_2\equiv (d^2-s_2)_2^\prime \Mod{4}$; hence,
\begin{align}
\label{s=2t}
    \int_{0\leq 2\nu_2(t)=\nu_2(s)} w_2^*(t)dt=\sum_{k=0}^\infty \frac{1}{2^{\frac{\nu_2(s)}{2}+k+3}}\sum_{\substack{{d\in (\Z/2^{k+3}\Z)^*}\\{d^2\equiv s_2 \Mod{2^k}}\\{d^2\not \equiv s_2 \Mod{2^{k+1}}}}}(d^2-s_2)_2^\prime w_2(d\cdot 2^\frac{\nu_2(s)}{2}).
\end{align}
From here, we consider various cases:

\begin{enumerate}
\item  $s_2 \equiv 3 \Mod{4}$: Let $d\in (\Z/2^{k+3}\Z)^*$ and suppose that $d^2\equiv s_2 \Mod{2^k}$, with $k\geq 2$. Then, $d^2\equiv s_2 \Mod{4}$. Under the assumption that $s_2\equiv 3 \Mod{4}$, we have that $d^2\equiv 3 \Mod{4}$, a contradiction, as all $d\in (\Z/2^{k+3}\Z)^*$ have squares equivalent to $1$ modulo $4$; and so, the sums 
\begin{align*}
\sum_{\substack{{d\in (\Z/2^{k+3}\Z)^*}\\{d^2\equiv s_2 \Mod{2^k}}\\{d^2\not \equiv s_2 \Mod{2^{k+1}}}}}(d^2-s_2)_2^\prime w_2(d\cdot 2^\frac{\nu_2(s)}{2})
\end{align*}
are empty for all $k\geq 2$. Similarly, there are no $d\in (\Z/8\Z)^*$ with $d^2\not \equiv 1 \Mod{2}$, so that the above sum is also empty for $k=0$. On the other hand, all $d\in (\Z/16\Z)^*$ are such that $d^2\equiv 1 \Mod{2},\not \equiv 3 \Mod{4}$; that is, 
\begin{align*}
    \sum_{\substack{{d\in (\Z/16\Z)^*}\\{d^2\equiv s_2 \Mod{2}}\\{d^2\not \equiv s_2 \Mod{4}}}}(d^2-s_2)_2^\prime w_2(d\cdot 2^\frac{\nu_2(s)}{2})
    =\sum_{d\in (\Z/16\Z)^*}(d^2-s_2)_2^\prime w_2(d\cdot 2^\frac{\nu_2(s)}{2}).
\end{align*}
It now follows that the only contribution to Equation \ref{s=2t}, when $s_2\equiv 3 \Mod{4}$, comes from $k=1$; in other words,
\begin{align*}
       \int_{0\leq 2\nu_2(t)=\nu_2(s)} w_2^*(t)dt
       = \frac{1}{2^{\frac{\nu_2(s)}{2}+4}}\sum_{d\in (\Z/16\Z)^*} (d^2-s_2)_2^\prime w_2(d\cdot 2^\frac{\nu_2(s)}{2}).
\end{align*}

By considering $(d^2-s_2)_2$ for $s_2\equiv 3,7,11,15 \Mod{16}$ and as $d$ varies over $(\Z/16\Z)^*$, we get that
\begin{align*}
    \int_{0\leq 2\nu_2(t)=\nu_2(s)} w_2^*(t)dt=
\left\{\def\arraystretch{1.2}%
\begin{array}{@{}l@{\quad}l@{}}
\left\{\def\arraystretch{1.2}%
\begin{array}{@{}l@{\quad}l@{}}
0 & \mbox{if $s_2\equiv 3 \Mod{8},$}\\
\frac{-1}{2^{\frac{\nu_2(s)}{2}+1}} & \mbox{if $s_2\equiv 7 \Mod{16},$}\\
\frac{1}{2^{\frac{\nu_2(s)}{2}+1}} & \mbox{if $s_2\equiv 15 \Mod{16},$}
\end{array}\right.
& \mbox{if $\nu_2(s)\equiv 0 \Mod{4}$,}\\
0 & \mbox{if $\nu_2(s)\equiv 2 \Mod{4}$}.
\end{array}\right.
\end{align*}

\item $s_2\equiv 5 \Mod{8}$: Similarly to the case above, let $d\in (\Z/2^{k+3}\Z)^*$ and suppose that $d^2\equiv s_2 \Mod{2^k}$, with $k\geq 3$. Then, $d^2\equiv s_2 \Mod{8}$. Under the assumption that $s_2\equiv 5 \Mod{8}$, we have that $d^2\equiv 5 \Mod{8}$, a contradiction, as all $d\in (\Z/2^{k+3}\Z)^*$ have squares equivalent to $1$ modulo $8$. So, the sums 
\begin{align*}
\sum_{\substack{{d\in (\Z/2^{k+3}\Z)^*}\\{d^2\equiv s_2 \Mod{2^k}}\\{d^2\not \equiv s_2 \Mod{2^{k+1}}}}}(d^2-s_2)_2^\prime w_2(d\cdot 2^\frac{\nu_2(s)}{2})
\end{align*}
are empty for all $k\geq 3$. Similarly, there are no $d\in (\Z/8\Z)^*$ (resp. $(\Z/16\Z)^*$) with $d^2\not \equiv 1 \Mod{2}$ (resp. $d^2\equiv 1\Mod{2},\not \equiv 1 \Mod{4}$), so that the above sums are also empty for $k=0,1$. On the other hand, all $d\in (\Z/32\Z)^*$ are such that $d^2\equiv 1 \Mod{4},\not \equiv 5 \Mod{8}$; that is, 
\begin{align*}
    \sum_{\substack{{d\in (\Z/32\Z)^*}\\{d^2\equiv s_2 \Mod{4}}\\{d^2\not \equiv s_2 \Mod{8}}}}(d^2-s_2)_2^\prime w_2(d\cdot 2^\frac{\nu_2(s)}{2})
    =\sum_{d\in (\Z/32\Z)^*}(d^2-s_2)_2^\prime w_2(d\cdot 2^\frac{\nu_2(s)}{2}).
\end{align*}
It now follows that the only contribution to Equation \ref{s=2t}, when $s_2\equiv 5 \Mod{8}$, comes from $k=2$; i.e.,
\begin{align*}
       \int_{0\leq 2\nu_2(t)=\nu_2(s)} w_2^*(t)dt
       = \frac{1}{2^{\frac{\nu_2(s)}{2}+5}}\sum_{d\in (\Z/32\Z)^*} (d^2-s_2)_2^\prime w_2(d\cdot 2^\frac{\nu_2(s)}{2}).
\end{align*}
By considering $(d^2-s_2)_2$ for $s_2\equiv 5,13 \Mod{16}$, $d\in (\Z/32\Z)^*$, it is also not hard to show that
\begin{align*}
\int_{0\leq 2\nu_2(t)=\nu_2(s)} w_2^*(t)dt=
\left\{\def\arraystretch{1.2}%
\begin{array}{@{}l@{\quad}l@{}}
0 & \mbox{if $\nu_2(s)\equiv 0 \Mod{4},$}\\
\left\{\def\arraystretch{1.2}%
\begin{array}{@{}l@{\quad}l@{}}
\frac{1}{2^{\frac{\nu_2(s)}{2}+2}} & \mbox{if $s_2\equiv 5 \Mod{16}$,}\\
\frac{1}{2^{\frac{\nu_2(s)}{2}+1}} & \mbox{if $s_2\equiv 13 \Mod{16}$,}\\
\end{array}\right.
& \mbox{if $\nu_2(s)\equiv 2 \Mod{4}$}.
\end{array}\right.
\end{align*}

\item $s_2\equiv 1 \Mod{8}$: In the case where $s_2\equiv 1 \Mod{8}$, we apply a less barbaric approach to computing $\int_{0\leq 2\nu_2(t)=\nu_2(s)}w_2^*(t)dt$. Firstly, notice that there are no $d\in (\Z/2^{k+3}\Z)^*$ such that $d^2\equiv 1 \Mod{2^k}\not \equiv 1 \Mod{2^{k+1}}$ for $k=0,1,2$; that is,
\begin{align*}
     \int_{0\leq 2\nu_2(t)=\nu_2(s)} w_2^*(t)dt=\sum_{k=3}^\infty \frac{1}{2^{\frac{\nu_2(s)}{2}+k+3}}\sum_{\substack{{d\in (\Z/2^{k+3}\Z)^*}\\{d^2\equiv s_2 \Mod{2^k}}\\{d^2\not \equiv s_2 \Mod{2^{k+1}}}}}(d^2-s_2)_2^\prime w_2(d\cdot 2^\frac{\nu_2(s)}{2}).
\end{align*}

Our next goal is to characterize all $d\in (\Z/2^{k+3}\Z)^*$ such that $d^2\equiv s_2 \Mod{2^k},\not \equiv s_2 \Mod{2^{k+1}}$, for $k\geq 3$. We begin by noting that all integers congruent to $1$ modulo $8$ admit a square root in $\Z_2$ (this follows from Hensel's Lemma). So, let $\pm \sqrt{s_2}$ denote the square roots of $s_2$ in $\Z_2$ and consider 
\begin{align}
\label{d}
d=d(\alpha_1,\alpha_2)&=\pm(\sqrt{s_2} +2^{k-1}(1+\alpha_1\cdot 2+ \alpha_2\cdot 2^2 +\alpha_3\cdot 2^3)) +2^{k+3}\Z_2\\
&\in (\Z_2/2^{k+3}\Z_2)^*\cong (\Z/2^{k+3}\Z)^*,
\end{align}
where $\alpha_i \in \{0,1\}, i=1,2,3$. Then, $d^2\equiv s_2 \Mod{2^k},\not \equiv s_2 \Mod{2^{k+1}}$. Moreover, 
\begin{align*}
    (d^2-s_2)_2
&\equiv 
\left\{\def\arraystretch{1.2}%
\begin{array}{@{}l@{\quad}l@{}}
2(1+2\alpha_1)+\sqrt{s_2}(1+2\alpha_1+4\alpha_2) \Mod{8} &\mbox{if $k=3,$}\\
4+\sqrt{s_2}(1+2\alpha_1+4\alpha_2) \Mod{8} & \mbox{if $k=4$,}\\
\sqrt{s_2}(1+2\alpha_1+4\alpha_2) \Mod{8} & \mbox{if $k\geq 5$}.
\end{array}\right.    
\end{align*}
\begin{rem}
The reason we label $d$ above as $d(\alpha_1,\alpha_2)$ will become apparent. Essentially, we only care for the values of $d,(d^2-s_2)_2$ modulo $8$, so that the value of $\alpha_3$ is irrelevant in our calculations: from Appendix A in \cite{BDD}, $w_2(t)$ is completely determined by $\nu_2(t)$ and $t_2,(t_2^2-s_2)_2\Mod{8}$.
\end{rem}
What's important to note is that the value of $(d^2-s_2)_2^\prime$ depends only on $\alpha_1$. Furthermore, the values of $(d^2-s_2)_2^\prime$ at $\alpha_1=0$ and $\alpha_1=1$ are negatives of one another! We claim further that Equation \ref{d} characterizes all $d\in(\Z/2^{k+3}\Z)^*$ such that $d^2\equiv s_2\Mod{2^k},\not\equiv s_2 \Mod{2^{k+1}}$: this follows from a simple counting argument. First note that there are exactly four $d\in (\Z/2^k \Z)^*$ such that $d^2\equiv s_2 \Mod{2^k}$, each of which lifts in exactly two ways to $d\in(\Z/2^{k+1}\Z)^*$ such that $d^2\equiv s_2 \Mod{2^k}$. Of these eight solutions, exactly four satisfy $d^2\equiv s_2 \Mod{2^{k+1}}$; that is, there are exactly four $d\in (\Z/2^{k+1}\Z)^*$ such that $d^2\equiv s_2\Mod{2^k},\not\equiv s_2\Mod{2^{k+1}}$, each of which lifts in exactly four ways to $d\in (\Z/2^{k+3}\Z)^*$ such that $d^2\equiv s_2\Mod{2^k},\not\equiv s_2\Mod{2^{k+1}}$.

By the preceding remarks, we may write
\begin{align*}
     \sum_{\substack{{d\in (\Z/2^{k+3}\Z)^*}\\{d^2\equiv s_2 \Mod{2^k}}\\{d^2\not \equiv s_2 \Mod{2^{k+1}}}}}(d^2-s_2)_2^\prime w_2(d\cdot 2^\frac{\nu_2(s)}{2})
\end{align*}
as
\begin{align*}     
     (2\chi_{k=3}(k)+\sqrt{s_2})^\prime
     \Biggr(
     &\Big(w_2(d(0,0)\cdot 2^\frac{\nu_2(s)}{2})+w_2(-d(0,0)\cdot 2^\frac{\nu_2(s)}{2})
     +w_2(d(0,1)\cdot 2^\frac{\nu_2(s)}{2})+w_2(-d(0,1)\cdot 2^\frac{\nu_2(s)}{2})\Big)\\
     -&\Big(w_2(d(1,0)\cdot 2^\frac{\nu_2(s)}{2})+w_2(-d(1,0)\cdot 2^\frac{\nu_2(s)}{2})
     +w_2(d(1,1)\cdot 2^\frac{\nu_2(s)}{2})+w_2(-d(1,1)\cdot 2^\frac{\nu_2(s)}{2})\Big)\Biggr).
\end{align*}

A case by case analysis then shows that, for $s_2\equiv 1 \Mod{8}$,
\begin{align*}
    \int_{0\leq 2\nu_2(t)=\nu_2(s)} w_2^*(t)dt
    =\left\{\def\arraystretch{1.2}%
\begin{array}{@{}l@{\quad}l@{}}
0 &\mbox{if $\nu_2(s)\equiv 0\Mod{4},$}\\
\frac{-1}{2^{\frac{\nu_2(s)}{2}+2}} & \mbox{if $\nu_2(s)\equiv 2 \Mod{4}.$}
\end{array}\right.  
\end{align*}

For the sake of completeness, we say a few more words. We deal with the case where $\nu_2(s) \equiv 0 \Mod{4}$, the case where $\nu_2(s)\equiv 2 \Mod{4}$ being eerily similar. Firstly, recall that $k\geq 3$. If $k\equiv 0,2,3,4 \Mod{6}, k\neq 2,3$, then $w_2(d\cdot 2^\frac{\nu_2(s)}{2})=1$ iff $d\equiv(d^2-s_2)_2 \Mod{4}$; in particular, $w_2(d\cdot 2^\frac{\nu_2(s)}{2})+w_2(-d\cdot 2^\frac{\nu_2(s)}{2})=0$ for all $d$. Therefore, the sums over $k\equiv 0,2,3,4 \Mod{6},k\neq 2,3,4$ are all equal to $0$. If $k\equiv 1,5 \Mod{6},k\neq 1,5$, then $w_2(d\cdot 2^\frac{\nu_2(s)}{2})=-1$ for all $d$; in this case,
\begin{align*}
&w_2(d(0,0)\cdot 2^\frac{\nu_2(s)}{2})+w_2(-d(0,0)\cdot 2^\frac{\nu_2(s)}{2})
     +w_2(d(0,1)\cdot 2^\frac{\nu_2(s)}{2})+w_2(-d(0,1)\cdot 2^\frac{\nu_2(s)}{2})\\
     &\;\;\;\;=w_2(d(1,0)\cdot 2^\frac{\nu_2(s)}{2})+w_2(-d(1,0)\cdot 2^\frac{\nu_2(s)}{2})
     +w_2(d(1,1)\cdot 2^\frac{\nu_2(s)}{2})+w_2(-d(1,1)\cdot 2^\frac{\nu_2(s)}{2}).
\end{align*}
Again, the sums over $k\equiv 1,5\Mod{6},k\neq 1,5$, are equal to $0$. For $k=3$, $w_2(d\cdot 2^\frac{\nu_2(s)}{2})=1$ iff $d\equiv 1 \Mod{4}$ and $d(d^2-s_2)_2\equiv 5,7 \Mod{8}$ or $d\equiv 3 \Mod{4}$ and $d(d^2-s_2)_2\equiv 3,5 \Mod{8}$. Since $d\equiv \pm \sqrt{s_2} \Mod{4}$ and since 
\begin{align*}
d(d^2-s_2)_2 
&\equiv \pm 
\left\{\def\arraystretch{1.2}%
\begin{array}{@{}l@{\quad}l@{}}
6\sqrt{s_2}+1 & \mbox{if $\alpha_1=0,\alpha_2=0$,}\\
6\sqrt{s_2}+5 & \mbox{if $\alpha_1=0,\alpha_2=1$,}\\
6\sqrt{s_2}+3 & \mbox{if $\alpha_1=1,\alpha_2=0$,}\\
6\sqrt{s_2}+7 & \mbox{if $\alpha_1=1,\alpha_2=1$,}
\end{array}\right.   
\end{align*}
it is easy to see that the sum at $k=3$ is also 0. Similarly, for the sum at $k=5$, $w_2(d\cdot 2^\frac{\nu_2(s)}{2})=1$ iff $d(d^2-s_2)_2\equiv 1,3,7 \Mod{8}$. In this case, 
\begin{align*}
d(d^2-s_2)_2 
&=\pm
 \left\{\def\arraystretch{1.2}%
\begin{array}{@{}l@{\quad}l@{}}
1 & \mbox{if $\alpha_1=0,\alpha_2=0$,}\\
5 & \mbox{if $\alpha_1=0,\alpha_2=1$,}\\
3 & \mbox{if $\alpha_1=1,\alpha_2=0$,}\\
7 & \mbox{ if $\alpha_1=1,\alpha_2=1$;}
\end{array}\right.
\end{align*}
in particular, the sum at $k=5$ is $0$.
\end{enumerate}

To summarize this subsection, 
\begin{align*}
\int_{0\leq 2\nu_2(t)=\nu_2(s)} w_2^*(t)dt
=
 \left\{\def\arraystretch{1.2}%
\begin{array}{@{}l@{\quad}l@{}}
0 & \mbox{if $\nu_2(s)\equiv 1 \Mod{2}$,}\\
 \left\{\def\arraystretch{1.2}%
\begin{array}{@{}l@{\quad}l@{}}
0 & \mbox{if $s_2\equiv 1,3,5 \Mod{8},$}\\
\frac{-1}{2^{\frac{\nu_2(s)}{2}+1}} & \mbox{if $s_2\equiv 7 \Mod{16},$}\\
\frac{1}{2^{\frac{\nu_2(s)}{2}+1}} & \mbox{if $s_2\equiv 15 \Mod{16},$}
\end{array}\right.
&\mbox{if $\nu_2(s)\equiv 0 \Mod{4},$}\\
 \left\{\def\arraystretch{1.2}%
\begin{array}{@{}l@{\quad}l@{}}
0 & \mbox{if $s_2\equiv 3 \Mod{4},$}\\
\frac{-1}{2^{\frac{\nu_2(s)}{2}+2}} & \mbox{if $s_2\equiv 1 \Mod{8},$}\\
\frac{1}{2^{\frac{\nu_2(s)}{2}+2}} & \mbox{if $s_2\equiv 5 \Mod{16},$}\\
\frac{1}{2^{\frac{\nu_2(s)}{2}+1}} & \mbox{if $s_2\equiv 13 \Mod{16},$}
\end{array}\right.
&\mbox{if $\nu_2(s)\equiv 2 \Mod{4}$}.
\end{array}\right.    
\end{align*}

Combining the results of the previous three subsections, 
\begin{prop}
\label{p=2}
\begin{align*}
&\int_{\Z_2}w_2^*(t)dt\\
&=\left\{\def\arraystretch{1.2}%
\begin{array}{@{}l@{\quad}l@{}}
0 & \mbox{if $\nu_2(s)\equiv 0 \Mod{2},$}\\
\frac{(-1)^\frac{\nu_2(s)-1}{2}}{2^\frac{\nu_2(s)+3}{2}}
\left\{\def\arraystretch{1.2}%
\begin{array}{@{}l@{\quad}l@{}}
1 & \mbox{if $s_2\equiv 1,7 \Mod{8},$}\\
-1 & \mbox{if $s_2 \equiv 3,5 \Mod{8},$}
\end{array}\right.
& \mbox{if $\nu_2(s)\equiv 1  \Mod{2},$}
\end{array}\right.\\
&+\left\{\def\arraystretch{1.2}%
\begin{array}{@{}l@{\quad}l@{}}
0 & \mbox{if $\nu_2(s)=0,$}\\
0 & \mbox{if $\nu_2(s)=1,$}\\
\frac{1}{4}
\cdot \left\{\def\arraystretch{1.2}%
\begin{array}{@{}l@{\quad}l@{}}
1 & \mbox{if $s_2\equiv 1 \Mod{4},$}\\
-2 & \mbox{if $s_2 \equiv 3 \Mod{4},$}
\end{array}\right.
&\mbox{if $\nu_2(s)=2,$}\\
0 & \mbox{if $\nu_2(s)=3,$}\\
\frac{1}{4} & \mbox{if $\nu_2(s)=4,$}\\
\frac{1-\chi_4(s_2)}{4} & \mbox{if $\nu_2(s)=5,$}\\
\frac{1}{16}
\cdot \left\{\def\arraystretch{1.2}%
\begin{array}{@{}l@{\quad}l@{}}
1 & \mbox{if $s_2\equiv 1 \Mod{4},$}\\
-2 & \mbox{if $s_2 \equiv 3 \Mod{4},$}
\end{array}\right. & \mbox{if $\nu_2(s)=6,$}\\
\frac{1}{3}\Biggr(\frac{2^{2\lceil \frac{7-j}{4}\rceil}}{2^{\frac{\nu_2(s)-j}{2}+2}}-1\Biggr)
+\frac{1}{2^\frac{\nu_2(s)-j}{2}}
\left\{\def\arraystretch{1.2}%
\begin{array}{@{}l@{\quad}l@{}}
1 & \mbox{if $\nu_2(s)\equiv 0 \Mod{4}$,}\\
1-\chi_4(s_2) & \mbox{if $\nu_2(s)\equiv 1 \Mod{4},$}\\
 \left\{\def\arraystretch{1.2}%
\begin{array}{@{}l@{\quad}l@{}}
\frac{1}{4} & \mbox{if $s_2\equiv 1 \Mod{4},$}\\
\frac{-1}{2}& \mbox{if $s_2 \equiv 3 \Mod{4},$}
\end{array}\right.
& \mbox{if $\nu_2(s)\equiv 2 \Mod{4},$}\\
0 & \mbox{if $\nu_2(s)\equiv 3 \Mod{4},$}
\end{array}\right.
& \mbox{if $\nu_2(s)\geq 7$,}
\end{array}\right.\\
&+\left\{\def\arraystretch{1.2}%
\begin{array}{@{}l@{\quad}l@{}}
0 & \mbox{if $\nu_2(s)\equiv 1 \Mod{2},$}\\
 \left\{\def\arraystretch{1.2}%
\begin{array}{@{}l@{\quad}l@{}}
0 & \mbox{if $s_2\equiv 1,3,5 \Mod{8},$}\\
\frac{-1}{2^{\frac{\nu_2(s)}{2}+1}} & \mbox{if $s_2\equiv 7 \Mod{16},$}\\
\frac{1}{2^{\frac{\nu_2(s)}{2}+1}} & \mbox{if $s_2\equiv 15 \Mod{16},$}
\end{array}\right.
&\mbox{if $\nu_2(s)\equiv 0 \Mod{4}$,}\\
 \left\{\def\arraystretch{1.2}%
\begin{array}{@{}l@{\quad}l@{}}
0 & \mbox{if $s_2\equiv 3\Mod{4},$}\\
\frac{-1}{2^{\frac{\nu_2(s)}{2}+2}} & \mbox{if $s_2\equiv 1 \Mod{8},$}\\
\frac{1}{2^{\frac{\nu_2(s)}{2}+2}} & \mbox{if $s_2\equiv 5 \Mod{16},$}\\
\frac{1}{2^{\frac{\nu_2(s)}{2}+1}} & \mbox{if $s_2\equiv 13 \Mod{16},$}
\end{array}\right.
&\mbox{if $\nu_2(s)\equiv 2 \Mod{4}$},
\end{array}\right.    
\end{align*}
where $j\in \{0,1,2,3\}$ is such that $\nu_2(s)\equiv j \Mod{4}$ and where $\chi_4$ is the non-principal character modulo $4$.
\end{prop}